\documentclass[12pt]{amsart}

\usepackage{amssymb}

\usepackage[pdftex]{hyperref} 
\usepackage[capitalize]{cleveref}

\newcommand{\pref}[1]{{\upshape(\ref{#1})}}

\newcommand{\csee}[1]{(see \cref{#1})}
\newcommand{\cseebelow}[1]{(see \cref{#1} below)}
\newcommand{\fullcref}[2]{\cref{#1}\pref{#1-#2}}
\newcommand{\fullCref}[2]{\Cref{#1}\pref{#1-#2}}
\newcommand{\fullcsee}[2]{(see \fullcref{#1}{#2}}
\renewcommand{\Cref}{\cref}

\newcommand{\normal}{\trianglelefteq}

\newcommand{\iso}{\cong}

\newcommand{\dir}{\overrightarrow}
\newcommand{\integer}{\mathbb{Z}}
\renewcommand{\natural}{\mathbb{N}}
\newcommand{\Cayd}{\mathop{\dir{\mathrm{Cay}}}}

\newcommand{\ah}{\overline{a}}
\newcommand{\bh}{\overline{b}}

\newcommand{\ato}{\stackrel{\textstyle a}{\to}}
\newcommand{\bto}{\stackrel{\textstyle b}{\to}}

\makeatletter
\newcommand{\noprelistbreak}{\smallskip\@nobreaktrue\nopagebreak} 
\makeatother

\usepackage{color}
\newcommand{\refnote}[1]{%
	\marginpar{%
		\color{blue}
		\vbox to 0pt{\vss
		$\begin{pmatrix} \text{note} \\ \text{\ref{#1}} \end{pmatrix}$%
		\vskip -11pt}}}
\newcommand{\refnotelower}[2]{%
	\marginpar{%
		\color{blue}
		\vbox to 0pt{\vskip #1pt
		$\begin{pmatrix} \text{note} \\ \text{\ref{#2}} \end{pmatrix}$%
		\vss}}}
\theoremstyle{definition}
\newtheorem{aid}[equation]{}
\newcommand{\oldendaid}{}
\let\oldendaid=\endaid
\renewcommand{\endaid}{\oldendaid\bigskip\hrule width\textwidth \bigbreak}

 \newcounter{case}
 \newenvironment{case}[1][\unskip]{\refstepcounter{case}\it
 \medskip \noindent \kern-1pt Case \thecase\ #1. }{\unskip\upshape}
 \renewcommand{\thecase}{\arabic{case}}
\crefformat{case}{Case~#2#1#3}  

 \newenvironment{claim}[1][\unskip]{\it
 \medskip \noindent \kern-1pt Claim. #1 }{\unskip\upshape}

\numberwithin{equation}{section}

\theoremstyle{plain}
\newtheorem{thm}[equation]{Theorem}
\newtheorem{prop}[equation]{Proposition}
\newtheorem{cor}[equation]{Corollary}
\newtheorem{lem}[equation]{Lemma}
\newtheorem*{lem*}{Lemma}

\theoremstyle{definition}
\newtheorem*{notation}{Notation}
\newtheorem*{defn}{Definition}

\newtheorem{eg}[equation]{Example}

\theoremstyle{remark}
\newtheorem{rem}[equation]{Remark}
\newtheorem{rems}[equation]{Remarks}

\newtheorem*{ack}{Acknowledgments}

\crefname{rems}{Remark}{Remarks}
\crefmultiformat{lem}{Lemmas~#2#1#3}{ and~#2#1#3}{, #2#1#3}{, and~#2#1#3}

\begin{document}

\title[On Cayley digraphs that do not have hamiltonian paths]{On Cayley digraphs that \\ do not have hamiltonian paths}


\setbox0\hbox to \textwidth{\hss \it Department of Mathematics and Computer Science,\hss}
\setbox1\hbox to \textwidth{\hss \it University of Lethbridge, Lethbridge, Alberta, T1K~3M4, Canada\hss}
\setbox2\hbox to \textwidth{\hss \textsf{Dave.Morris@uleth.ca, http://people.uleth.ca/\!$\sim$dave.morris/}\hss}
\author[Dave Witte Morris]{Dave Witte Morris \\ \vskip 0pt
\copy0 \smallskip \copy1 \medskip \copy2
}

\begin{abstract}
We construct an infinite family $\bigl\{ \Cayd(G_i;a_i,b_i) \bigr\}$ of connected, $2$-generated Cayley digraphs that do not have hamiltonian paths, such that the orders of the generators $a_i$ and~$b_i$ are unbounded. 
We also prove that if $G$ is any finite group with $|[G,G]| \le 3$, then every connected Cayley digraph on~$G$ has a hamiltonian path (but the conclusion does not always hold when $|[G,G]| = 4$ or~$5$).
\end{abstract}

\date{\today} 

\maketitle

\section{Introduction}

\begin{defn}
For a subset~$S$ of a finite group~$G$, the \emph{Cayley digraph} $\Cayd(G;S)$\refnote{S=ab}
 is the directed graph whose vertices are the elements of~$G$, and with a directed edge $g \to gs$ for every $g \in G$ and $s \in S$. The corresponding \emph{Cayley graph} is the underlying undirected graph that is obtained by removing the orientations from all the directed edges.
\end{defn}

It has been conjectured that every (nontrivial) connected Cayley graph has a hamiltonian cycle. (See the bibliography of \cite{KutnarEtAl-SmallOrder} for some of the  literature on this problem.) This conjecture does not extend to the directed case, because there are many examples of connected Cayley digraphs that do not have hamiltonian cycles. In fact, infinitely many Cayley digraphs do not even have a hamiltonian path:

\begin{prop}[{attributed to J.\,Milnor \cite[p.~201]{Nathanson-PartProds}}] 
	 \label{23NoHP}
Assume the finite group~$G$ is generated by two elements $a$ and~$b$, such that $a^2 = b^3 = e$.\refnote{ProveMilnor}  If\/ $|G| \ge 9  |ab^2|$, then the Cayley digraph\/ $\Cayd(G; a,b)$ does not have a hamiltonian path.
\end{prop}

%

The examples in the above \lcnamecref{23NoHP} 
are very constrained, because the order of one generator must be exactly~$2$, and the order of the other generator must be exactly~$3$. In this note, we provide an infinite family of examples in which the orders of the generators are not restricted in this way. In fact, $a$ and~$b$ can both be of arbitrarily large order:

\begin{thm} \label{NewNonTraceable}
For any $n \in \natural$, there is a connected Cayley digraph\/ $\Cayd(G;a,b)$, such that 
\noprelistbreak
	\begin{enumerate}
	\item $\Cayd(G;a,b)$ does not have a hamiltonian path,
	and
	\item $a$ and~$b$ both have order greater than~$n$. 
	\end{enumerate}
Furthermore, if $p$ is any prime number such that $p > 3$ and $p \equiv 3 \pmod{4}$, then we may construct the example so that the commutator subgroup of~$G$ has order~$p$. More precisely, $G = \integer_m \ltimes \integer_p$ is a semidirect product of two cyclic groups, so $G$ is metacyclic.
\end{thm}

\begin{rems} \
	\noprelistbreak
	\begin{enumerate} \itemsep=\smallskipamount
	\item The above results show that connected Cayley digraphs on solvable groups do not always have hamiltonian paths. On the other hand, it is an open question whether connected Cayley digraphs on \emph{nilpotent} groups always have hamiltonian paths. (See \cite{Morris-2gen} for recent results on the nilpotent case.)
	\item The above results always produce a digraph with an even number of vertices. Do there exist infinitely many connected Cayley digraphs of odd order that do not have hamiltonian paths?
	\item We conjecture that the assumption ``$p \equiv 3 \pmod{4}$'' can be eliminated from the statement of \cref{NewNonTraceable}. On the other hand, it is necessary to require that $p > 3$ \csee{GGle3}.
	\item If $G$ is abelian, then it is easy to show that every connected Cayley digraph on~$G$ has a hamiltonian path. However, some abelian Cayley digraphs do not have a hamiltonian cycle. See \cref{NonhamAbelian} for more discussion of this.
	\item The proof of \cref{NewNonTraceable} appears in \cref{MainThmPfSect}, after some preliminaries in \cref{PrelimSect}.
	\end{enumerate}
\end{rems}

\section{Preliminaries} \label{PrelimSect}

We recall some standard notation, terminology, and basic facts.

\begin{notation} 
Let $G$ be a group, and let $H$ be a subgroup of~$G$. (All groups in this paper are assumed to be finite.)
\noprelistbreak
	\begin{itemize}
	\item $e$ is the identity element of~$G$.
	\item $x^g = g^{-1} x g$, for $x,g \in G$.
	\item We write $H \normal G$ to say that $H$ is a \emph{normal} subgroup of~$G$.
	\item $H^G = \langle\, h^g \mid h \in H, \, g \in G \,\rangle$ is the \emph{normal closure} of~$H$ in~$G$, so $H^G \normal G$.
	\end{itemize}
\end{notation}

\begin{defn}
Let $S$ be a subset of the group~$G$.
\noprelistbreak
	\begin{itemize}
	\item $H = \langle S S^{-1} \rangle$ is the \emph{arc-forcing subgroup}, where $S S^{-1} = \{\, s t^{-1} \mid s,t \in S \,\}$.
	\item For any $a \in S$, $a^{-1} H$ is called the \emph{terminal coset}. (This is independent of the choice of~$a$.)\refnote{TermIndep}
	\item Any left coset of~$H$ that is not the terminal coset is called a \emph{regular coset}.
	\item For $g \in G$ and $s_1,\ldots,s_n \in S$, we use $[g](s_i)_{i=1}^n$ to denote the walk in $\Cayd(G;S)$ that visits (in order) the vertices
		$$ g, \ g s_1, \ g s_1 s_2, \ \ldots, \  g s_1 s_2 \cdots  s_n .$$
	We usually omit the prefix $[g]$ when $g = e$. Also, we often abuse notation when sequences are to be concatenated. For example,
		$$ \bigl( a^4, (s_i)_{i=1}^3, t_j \bigr)_{j=1}^2
		= (a, a, a, a, s_1, s_2, s_3, t_1, a, a, a, a,  s_1, s_2, s_3, t_2 ) .$$ 
	\end{itemize}
\end{defn}

\begin{rems} \label{ArcForcingRems} \ 
\noprelistbreak
	\begin{enumerate}
	\item  \label{ArcForcingRems-Sg}
	It is important to note that $\langle S S^{-1} \rangle \subseteq \langle Sg \rangle$, for every $g \in G$. Furthermore, we have $\langle S S^{-1} \rangle = \langle Sa^{-1} \rangle$, for every $a \in S$.\refnote{Sa}
	\item \label{ArcForcingRems-conj}
	It is sometimes more convenient to define the arc-forcing subgroup to be $\langle S^{-1} S \rangle$, instead of $\langle S S^{-1} \rangle$. (For example, this is the convention used in \cite[p.~42]{Morris-2gen}.) The difference is minor, because the two subgroups are conjugate: for any $a \in S$, we have\refnotelower0{aS}
	$$ \langle S^{-1} S \rangle = \langle a^{-1} S \rangle = \langle S a^{-1}  \rangle^ a = \langle S S^{-1} \rangle^a.$$
	\end{enumerate}
\end{rems}

\begin{defn}
Suppose $L$~is a hamiltonian path in a Cayley digraph $\Cayd(G; S)$, and $s \in S$.
\noprelistbreak
	\begin{itemize}
	\item A vertex $g \in G$ \emph{travels by~$s$} if $L$ contains the directed edge $g \to gs$.
	\item A subset $X$ of~$G$ \emph{travels by~$s$} if every element of~$X$ travels by~$s$.
	\end{itemize}
\end{defn}

\begin{lem}[{}{Housman \cite[p.~42]{Housman-enumeration}}] \label{HousmanThm}
Suppose $L$ is a hamiltonian path in $\Cayd(G;a,b)$, with initial vertex~$e$, and let $H = \langle a b^{-1} \rangle$ be the arc-forcing subgroup.\refnote{HousmanHAid}
 Then:
\noprelistbreak
	\begin{enumerate}
	\item \label{HousmanThm-terminal}
	The terminal vertex of~$L$ belongs to the terminal coset $a^{-1} H$.\refnotelower{-10}{HousmanAid}
	\item \label{HousmanThm-regular}
	Each regular coset either travels by~$a$ or travels by~$b$.
	\end{enumerate}
\end{lem}

\section{\texorpdfstring{Proof of \cref{NewNonTraceable}}{Proof of the Main Theorem}} 
 \label{MainThmPfSect}

Let 
\noprelistbreak
	\begin{itemize} \itemsep=\smallskipamount
	\item $\alpha$ be an even number that is relatively prime to $(p-1)/2$, with $\alpha > n$,
	\item $\beta$ be a multiple of $(p-1)/2$ that is relatively prime to~$\alpha$, with $\beta > n$,
	\item $\ah$ be a generator of~$\integer_\alpha$, 
	\item $\bh$~be a generator of~$\integer_\beta$,
	\item $z$~be a generator of~$\integer_p$,
	\item $r$ be a primitive root modulo~$p$,
	\item $G = (\integer_\alpha \times \integer_\beta) \ltimes \integer_p$, where $z^{\ah} = z^{-1}$ and $z^{\bh} = z^{r^2}$,
	\item $a = \ah z$, so $|a| = \alpha$, and $a$ inverts~$\integer_p$,
	\item $b = \bh z$, so $|b| = \beta$, and $b$ acts on $\integer_p$ via an automorphism of order $(p-1)/2$,
	and
	\item $H = \bigl\langle a b^{-1} \bigr\rangle = \bigl\langle \ah \, \bh^{-1} \bigr\rangle = \integer_\alpha \times \integer_\beta$.\refnote{SetupH}
	\end{itemize}
Suppose $L$ is a hamiltonian path in $\Cayd(G;a,b)$. This will lead to a contradiction.

It is well known (and easy to see) that Cayley digraphs are vertex-transitive,\refnote{verttrans}
so there is no harm in assuming that the initial vertex of~$L$ is~$e$.
Note that:
\noprelistbreak
	\begin{itemize}
	\item the terminal coset is $a^{-1} H = z^{-1} H$,\refnote{TermCosetz/MinOne}
	and
	\item since $p \equiv 3 \pmod{4}$, 
	we have $ \integer_p^\times = \langle -1 ,  r^2 \rangle$.
	\end{itemize}

\setcounter{case}{0}

\begin{case}
Assume at most one regular coset travels by~$a$ in~$L$.
\end{case}
Choose $z' \in \integer_p$, such that $z' H$\refnote{AllzH}
is a regular coset, and assume it is the coset that travels by~$a$, if such exists. 
	
For $g \in G$, let
	$$ \mathcal{B}_g = \{\, g b^k H \mid k \in \integer \,\} .$$
Letting $p' = (p-1)/2$, we have 
	$$ (r^2)^{p'-1} + (r^2)^{p'-2} + \cdots + (r^2)^{1} + 1
	= \frac{(r^2)^{p'}-1}{r^2-1} 
	= \frac{r^{p-1} -1}{r^2-1} 
	\equiv 0 \pmod{p} ,$$
so
	$$b^{(p-1)/2} = (\bh z)^{p'} = \bh^{p'} z^{(r^2)^{p'-1} + (r^2)^{p'-2} + \cdots + (r^2)^{1} + 1} 
	= \bh^{p'} \in \integer_\alpha \times \integer_\beta = H 
	.$$
Therefore $\# \mathcal{B}_e \le (p-1)/2 \le p-2$, so we can choose two cosets $z^i H$ and $z^j H$ that do not belong to~$\mathcal{B}_e$.

Recall that, by definition, $z' H$ is not the terminal coset $z^{-1} H$, so $z' z$ is a nontrivial element of~$\integer_p$. Then, since $\integer_p^\times = \langle -1 , r^2 \rangle$, we can choose some $h \in \langle \ah , \bh \rangle = H$, such that $(z^{j-i})^h = z'z$. 
Now, since
	$$ z^i H, z^j H \notin \mathcal{B}_e ,$$
and 
	$$z^{-1} h^{-1} z^{j-i} \in z^{-1} (z^{j-i})^h H = z^{-1} (z'z) H = z' H ,$$
we may multiply on the left by $g = z^{-1}h^{-1} z^{-i}$ to see that
	$$ z^{-1} H,  z' H \notin \mathcal{B}_g .$$
Therefore, no element of~$\mathcal{B}_g$ is either the terminal coset or the regular coset that travels by~$a$. This means that every coset in $\mathcal{B}_g$ travels by~$b$, so $L$ contains the cycle 
	$ [g](b^\beta)$,
which contradicts the fact that $L$ is a (hamiltonian) path. 

\begin{case} \label{TwoCosetsCase}
Assume at least two regular cosets travel by~$a$ in~$L$.
\end{case}
Let $z^i H$ and $z^j H$ be two regular cosets that both travel by~$a$. Since $\integer_p^\times = \langle -1 , r^2 \rangle$, we can choose some $h \in \langle \ah , \bh \rangle = H$, such that $(z^{-1})^h = z^{j-i}$. 

Note that $z^i h^{-1} a^k$ travels by~$a$, for every $k \in \integer$:
\noprelistbreak
	\begin{itemize} \itemsep=\smallskipamount
	\item If $k = 2 \ell$ is even, then 
		$$ a^k = (\ah z)^{2\ell} 
		= \bigl( \ah z \ah z \bigr)^\ell 
		= \bigl( \ah^2 z^{\ah} z \bigr)^\ell 
		= \bigl( \ah^2 z^{-1} z \bigr)^\ell 
		= \ah^{2\ell} 
		\in H
		,$$ 
	so $z^i h^{-1} a^k \in z^i H$ travels by~$a$.

	\item If $k = 2 \ell + 1$ is odd, then 
	$$ a^k 
	= (\ah z)^{2\ell+1} 
	= (\ah z)^{2\ell} (\ah z) 
	= \ah^{2\ell} (\ah z) 
	= \ah^k z
	,$$
so
	$$z^i h^{-1} a^k 
	= z^i h^{-1} (\ah^k z)
	= z^i h^{-1} z^{-1} \ah^k
	= z^i (z^{-1})^h h^{-1} \ah^k
	\in z^i (z^{j-i}) H
	 = z^j H
	$$
travels by~$a$.
\end{itemize}
Therefore $L$ contains the cycle $[z^i h^{-1}](a^\alpha)$, which contradicts the fact that $L$ is a (hamiltonian) path.
\qed

%
%

\section{Cyclic commutator subgroups of very small order} \label{VerySmallSect}

It is known that if $|[G,G]| = 2$, then every connected Cayley digraph on~$G$ has a hamiltonian path. (Namely, we have $[G,G] \subseteq Z(G)$, so $G$ is nilpotent,\refnote{CentralIsNilp}
and the conclusion therefore follows from \fullcref{2gen}{prime} below.) In this \lcnamecref{VerySmallSect}, we prove the same conclusion when $|[G,G]| = 3$. We also provide counterexamples to show that the conclusion is not always true when $|[G,G]| = 4$ or $|[G,G]| = 5$.

We begin with several lemmas. The first three each provide a way to convert a hamiltonian path in a Cayley digraph on an appropriate subgroup of~$G$ to a hamiltonian path in a Cayley digraph on all of~$G$.

\begin{lem} \label{MayAssumeS=2}
Assume
\noprelistbreak
	\begin{itemize}
	\item $G$ is a finite group, such that\/ $[G,G] \iso \integer_{p^k}$, where $p$ is prime and $k \in \natural$,
	\item $S$ is a generating set for~$G$,
	\item $a,b \in S$, such that $\langle [a,b] \rangle = [G,G]$,
	and
	\item $N = \langle a,b \rangle$.
	\end{itemize}
If\/ $\Cayd(N;a,b)$ has a hamiltonian path, then\/ $\Cayd(G;S)$ has a hamiltonian path.
\end{lem}

\begin{proof}
Since $[G,G] \subseteq N$, we know that $G/N$ is an abelian group, so there is a hamiltonian path $(s_i)_{i=1}^m$ in $\Cayd(G/N;S)$ \cseebelow{abelianpath}.
Also, by assumption, there is a hamiltonian path $(t_j)_{j=1}^n$ in $\Cayd(N;a,b)$. Then 
	$$ \Bigl( \bigl( (t_j)_{j=1}^n, s_i \bigr)_{i=1}^m , (t_j)_{j=1}^n \Bigr)$$
is a hamiltonian path in $\Cayd(G;S)$.\refnote{MayAssumeS=2Aid}
\end{proof}

\begin{defn}
If $K$ is a subgroup of~$G$, then $K \backslash {\Cayd(G;S)}$ denotes the digraph whose vertices are the right cosets of~$K$ in~$G$, and with a directed edge $Kg \to Kgs$ for each $g \in G$ and $s \in S$. Note that $K \backslash {\Cayd(G;S)} = \Cayd(G/K; S)$ if $K \normal G$.
\end{defn}

\begin{lem}[``Skewed-Generators Argument,'' {cf.\ \cite[Lem.~2.6]{Morris-2gen}, \cite[Lem.~5.1]{Witte-pgrp}}] \label{SkewedGens}
Assume:
\noprelistbreak
	\begin{itemize}
	\item $S$ is a generating set for the group~$G$,
	\item $K$ is a subgroup of~$G$, such that every connected Cayley digraph on~$K$ has a hamiltonian path,
	\item $(s_i)_{i=1}^n$ is a hamiltonian cycle in $K \backslash {\Cayd(G;S)}$,
	and
	\item $\langle S s_2 s_3 \cdots s_n \rangle = K$.
	\end{itemize}
Then\/ $\Cayd(G;S)$ has a hamiltonian path.
\end{lem}

\begin{proof}
Since $\langle S s_2 s_3 \cdots s_n \rangle = K$, we know that $\Cayd(K; S s_2 s_3 \cdots s_n)$ is connected, so, by assumption, it has a hamiltonian path $(t_j s_2 s_3 \cdots s_n)_{j=1}^m$. Then
	$$ \Bigl( \bigl( t_j , (s_i)_{i=2}^n \bigr)_{j=1}^{m-1} , t_m, (s_i)_{i=2}^{n-1} \Bigr) $$
is a hamiltonian path in $\Cayd(G;S)$.\refnote{SkewedGensAid}
\end{proof}

\begin{lem} \label{GGPrimeHG}
Assume:
\noprelistbreak
	\begin{itemize}
	\item $S$ is a generating set of~$G$, with arc-forcing subgroup $H = \langle S S^{-1} \rangle$, 
	\item there is a hamiltonian path in every connected Cayley digraph on $H^G$,
	and
	\item either $H = H^G$, or $H$~is contained in a \underline{unique} maximal subgroup of~$H^G$.
	\end{itemize}
Then\/ $\Cayd(G;S)$ has a hamiltonian path.
\end{lem}

\begin{proof}
It suffices to show
	\begin{align} \tag{$*$} \label{Ssss=HG}
	\begin{matrix}
	 \text{there exists a hamiltonian cycle $(s_i)_{i=1}^n$ in $\Cayd(G/H^G; S)$,}
	 \\[3pt] \text{such that $H^G = \langle Ss_2 \cdots s_n \rangle$}
	 , \end{matrix}
	 \end{align}
for then \cref{SkewedGens} provides the desired hamiltonian path in $\Cayd(G;S)$.

If $H^G = H$, then every hamiltonian cycle in $\Cayd(G/H^G; S)$ satisfies \pref{Ssss=HG} \fullcsee{ArcForcingRems}{Sg}. Thus, we may assume $H^G \neq H$, so, by assumption, $H$ is contained in a unique maximal subgroup~$M$ of~$H^G$. 
Since $H^G$ is generated by conjugates of~$S^{-1}S$ \fullcsee{ArcForcingRems}{conj}, there exist $a,b,c \in S$, such that $(a^{-1}b)^c \notin M$.

We may also assume $H^G \neq G$ (since, by assumption, every Cayley digraph on~$H^G$ has a hamiltonian path), so, letting $n = |G:H^G| \ge 2$, we have the two hamiltonian cycles $(a^{n-1},c)$ and $(a^{n-2},b,c)$ in $\Cayd(G/H^G; S)$.\refnote{GGPrimeHGan}
 Since 
	$$(a^{n-1} c)^{-1} (a^{n-2}bc) = (a^{-1}b)^c \notin M ,$$
the two products $a^{n-1} c$ and $a^{n-2}bc$ cannot both belong to~$M$. Hence, either $(a^{n-1},c)$ or $(a^{n-2},b,c)$ is a hamiltonian cycle $(s_i)_{i=1}^n$  in $\Cayd(G/H^G; S)$, such that $s_1s_2 \cdots s_n \notin M$. Since $M$ is the unique maximal subgroup of~$H^G$ that contains~$H$, this implies\refnotelower0{GGPrimeHGM}
	$$ H^G = \langle H, s_1s_2 \cdots s_n \rangle 
	= \langle S s_2 s_3 \cdots s_n \rangle ,
	$$
as desired.
\end{proof}

The final hypothesis of the preceding 
\lcnamecref{GGPrimeHG} 
is automatically satisfied when $[G,G]$ is cyclic of prime-power order:

\begin{lem} \label{UniqueMax}
If\/ $[G,G]$ is cyclic of order~$p^k$, where $p$~is prime, and $H$ is any subgroup of~$G$, then either $H = H^G$, or $H$~is contained in a unique maximal subgroup of~$H^G$.
\end{lem}

\begin{proof}
Note that the normal closure $H^G$ is the (unique) smallest normal subgroup of~$G$ that contains~$H$.\refnote{UniqueMaxClosure} Therefore $H^G \subseteq H \, [G,G]$ (since $H \, [G,G]$ is normal in~$G$). This implies that if $M$ is any proper subgroup of~$H^G$ that contains~$H$, then\refnotelower0{UniqueMaxHK}
	$$M = H \cdot \bigl( M \cap [G,G] \bigr) \subseteq H \cdot \bigl( H^G \cap [G,G] \bigr)^p .$$
Therefore, $H \cdot \bigl( H^G \cap [G,G] \bigr)^p$ is the unique maximal subgroup of~$H^G$ that contains~$M$.
\end{proof}

The following known result handles the case where $G$ is nilpotent:

\begin{thm}[Morris \cite{Morris-2gen}] \label{2gen}
Assume $G$ is nilpotent, and $S$ generates~$G$. If either
\noprelistbreak
	\begin{enumerate}
	\item \label{2gen-2}
	$\#S \le 2$,
	or
	\item\label{2gen-prime}
	$|[G,G]| = p^k$, where $p$~is prime and $k \in \natural$,
	\end{enumerate}
then $\Cayd(G;S)$ has a hamiltonian path.
\end{thm}

We now state the main result of this \lcnamecref{VerySmallSect}:

\begin{thm} \label{OnlyInvertHP}
Suppose
\noprelistbreak
	\begin{itemize}
	\item $[G,G]$ is cyclic of prime-power order,
	and
	\item every element of~$G$ either centralizes\/ $[G,G]$ or inverts it.
	\end{itemize}
Then every connected Cayley digraph on~$G$ has a hamiltonian path.
\end{thm}

\begin{proof}
Let $S$ be a generating set for~$G$. Write $[G,G] = \integer_{p^k}$ for some $p$ and~$k$. Since every minimal generating set of $\integer_{p^k}$ has only one element, there exist $a,b \in S$, such that $\langle [a,b] \rangle = [G,G]$.\refnote{OnlyInvertHPGG}
Then $\langle a,b \rangle$ is normal in~$G$ (since it contains $[G,G]$), so, by \cref{MayAssumeS=2}, we may assume $S = \{a,b\}$. 

Let $H = \langle ba^{-1} \rangle$ be the arc-forcing subgroup. 
We may assume $H^G = G$, for otherwise we could assume, by induction on~$|G|$, that every connected Cayley digraph on $H^G$ has a hamiltonian path, and then \cref{GGPrimeHG} would apply (since \cref{UniqueMax} verifies the remaining hypothesis).
So 
	$$H \, \integer_{p^k}  =  H \, [G,G] \supset H^G = G .$$

If $a$ and~$b$ both invert~$\integer_{p^k}$, then $H = \langle b a^{-1} \rangle$ centralizes $\integer_{p^k} = [G,G]$, so $G$ is nilpotent. Then \cref{2gen} applies. 

Therefore, we may now assume that $a$ does not invert~$\integer_{p^k}$. Then, by assumption, $a$ centralizes~$\integer_{p^k}$. Let $n = |G:H|$, and write $a = \ah z$, where $\ah \in H$ and $z \in \integer_{p^k}$. Then $a = \ah z \in H z$ and $b = (ba^{-1}) (\ah z) \in H z$. Since $\langle a,b \rangle = G$, this implies $H \langle z \rangle = G$.\refnote{OnlyInvertHPHz}
Therefore
	$$[H](a^n) = [H, Hz, Hz^2,\ldots,Hz^{n-1}, H]$$
is a hamiltonian cycle in $H \backslash {\Cayd(G;S)}$, so \cref{SkewedGens} applies. 
\end{proof}

\begin{cor} \label{GGle3}
If\/ $|[G,G]| \le 3$, or\/ $[G,G] \iso \integer_4$, then every connected Cayley digraph on~$G$ has a hamiltonian path.
\end{cor}

\begin{proof}
\Cref{OnlyInvertHP} applies, because inversion is the only nontrivial automorphism of $\{e\}$, $\integer_2$, $\integer_3$, or~$\integer_4$.
\end{proof}

\begin{rem}[{}{\cite[p.~266]{HolsztynskiStrube}}]
In the statement of \cref{GGle3}, the assumption that $[G,G] \iso \integer_4$ cannot be replaced with the weaker assumption that $|[G,G]| = 4$. For a counterexample, let $G = A_4 \times \integer_2$. Then $|[G,G]| = 4$, but it can be shown without much difficulty that $\Cayd(G; a,b)$ does not have a hamiltonian path when $a = \bigl( (1\, 2)(3 \, 4), 1 \bigr)$ and $b = \bigl( (1\, 2 \, 3), 0 \bigr)$.\refnote{A4xZ2NoPath} 
\end{rem}

Here is a counterexample when $|[G,G]| = 5$:

\begin{eg} \label{G'=5NoHP}
Let $G = \integer_{12} \ltimes \integer_5 = \langle h \rangle \ltimes \langle z \rangle$, where $z^h = z^3$. Then $|[G,G]| = 5$, and the Cayley digraph $\Cayd(G; h^2z, h^3z)$ is connected, but does not have a hamiltonian path.
\end{eg}

\begin{proof}
A computer search can confirm the nonexistence of a hamiltonian path very quickly, but, for completeness, we provide a human-readable proof.

Let $a = h^2z = z^4 h^2$ and $b = h^3z = z^3 h^3$. The argument in \cref{TwoCosetsCase} of the proof of \cref{NewNonTraceable} shows that no more than one regular coset travels by~$a$ in any hamiltonian path. On the other hand, since a hamiltonian path cannot contain any cycle of the form $[g](b^4)$, we know that at least $\bigl\lfloor \bigl(|G| - 1 \bigr)/4 \bigr\rfloor = 14$ vertices must travel by~$a$. Since $|ab^{-1}| = 12 < 14$, this implies that some regular coset travels by~$a$. So exactly one regular coset travels by~$a$ in any hamiltonian path.

For $0 \le i \le 3$ and $0\le m \le 11$, let $L_{i,m}$ be the spanning subdigraph of $\Cayd(G; a,b)$ in which:
\noprelistbreak
	\begin{itemize}
	\item all vertices have outvalence~$1$, except $b^{-1}(a b^{-1})^m = z^4 h^{9-m}$, which has outvalence~$0$,
	\item the vertices in the regular coset $z^i H$ travel by~$a$,
	\item a vertex $b^{-1}h^{-j} = z^4 h^{9-j}$ in the terminal coset travels by~$a$ if $0 \le j < m$,
	and
	\item all other vertices travel by~$b$.
	\end{itemize}
An observation of D.\,Housman \cite[Lem.~6.4(b)]{CurranWitte} tells us that if $L$ is a hamiltonian path from $e$ to $b^{-1}(a b^{-1})^m$, in which $z^i H$ is the regular coset that travels by~$a$, then $L = L_{i,m}$.\refnote{TermCosetTravels} 
Thus, from the conclusion of the preceding paragraph, we see that every hamiltonian path (with initial vertex~$e$) must be equal to $L_{i,m}$, for some~$i$ and~$m$.

However, $L_{i,m}$ is not a (hamiltonian) path. More precisely, for each possible value of~$i$ and~$m$, the following list displays a cycle that is contained in~$L_{i,m}$:
\noprelistbreak
	\begin{itemize}
	
	\item  if $i = 0$ and $0 \le m \le 8$: 
		$$z^2 h^3 \bto z h^6 \bto z^3 h^9 \bto z^4 \bto z^2 h^3 ;$$
	\item if $i = 0$ and $9 \le m \le 11$: 
		$$h^2 \ato z h^4 \bto z^4 h^7 \ato z h^9 \bto z^2 \bto h^3 \ato z^2 h^5 \bto z^3 h^8 \bto z h^{11} \bto h^2 ;$$
		 
	\item if $i = 1$ and $0 \le m \le 7$: 
		$$h^4 \bto z^3 h^7 \bto z^2 h^{10} \bto z^4 h \bto h^4 ;$$  
	\item if $i = 1$ and $8 \le m \le 11$: 
		$$h \bto z h^4 \ato h^6 \bto z^2 h^9 \bto z^3 \bto z h^3 \ato z^3 h^5 \bto z^4 h^8 \ato z^3 h^{10} \bto h ;$$
		 
	\item if $i = 2$ and $0 \le m \le 9$: 
		$$h^5 \bto z h^8 \bto z^4 h^{11} \bto z^3 h^2 \bto h^5 ;$$   
	\item if $i = 2$ and $10 \le m \le 11$: 
		$$z^2 h^3 \ato z^4 h^5 \ato z^2 h^7 \ato z^4 h^9 \ato z^2 h^{11} \ato z^4 h \ato z^2 h^3 ; $$
		 
	\item if $i = 3$ and $0 \le m \le 10$: 
		$$h^7 \bto z^4 h^{10} \bto z h \bto z^2 h^4 \bto h^7 ;$$
	\item if $i = 3$ and $m = 11$: 
		$$z^3 h^2 \ato z^4 h^4 \ato z^3 h^6 \ato z^4 h^8 \ato z^3 h^{10} \ato z^4 \ato z^3 h^2 .$$		 
	 \end{itemize} 
Since $L_{i,m}$ is never a hamiltonian path, we conclude that $\Cayd(G;a,b)$ does not have a hamiltonian path.
\end{proof}

\section{Nonhamiltonian Cayley digraphs on abelian groups} \label{NonhamAbelian}

When $G$ is abelian, it is easy to find a hamiltonian path in $\Cayd(G;S)$:

\begin{prop}[{}{\cite[Thm.~3.1]{HolsztynskiStrube}}] \label{abelianpath}
Every connected Cayley digraph on any abelian group has a hamiltonian path.\refnote{abelianpathaid}
\end{prop}

On the other hand, it follows from \fullcref{HousmanThm}{regular} that sometimes there is no hamiltonian cycle:

\begin{prop}[Rankin {\cite[Thm.~4]{Rankin-Camp}}] \label{Rankin2gen}
Assume $G = \langle a, b \rangle$ is abelian. Then there is a hamiltonian cycle in $\Cayd(G; a,b)$ if and only if\refnote{RankinAid}
 there exist $k,\ell \ge 0$, such that $\langle a^k b^\ell \rangle = \langle a b^{-1} \rangle$, and $k + \ell = |G : \langle a b^{-1} \rangle|$.
\end{prop}

\begin{eg}
If $\gcd(a, n) > 1$ and $\gcd(a+1, n) > 1$, then $\Cayd \bigl( \integer_n ; a, a+1 \bigr)$ does not have a hamiltonian cycle.\refnote{EasyNoHam}
\end{eg}

The non-hamiltonian Cayley digraphs provided by \cref{Rankin2gen} are $2$-generated. A few $3$-generated examples are also known. Specifically, the following result lists (up to isomorphism) the only known examples of connected, non-hamiltonian Cayley digraphs $\Cayd(G; S)$, such that $\#S > 2$ (and $e \notin S$):

\begin{thm}[Locke-Witte {\cite{LockeWitte}}] \label{LockeWitte}
The following Cayley digraphs do not have hamiltonian cycles:
	\begin{enumerate}
	\item  \label{LockeWitte-12k}
	$\Cayd(\integer_{12k}; 6k, 6k+2, 6k+3)$, for any $k \in \integer^+$,
	and
	\item  \label{LockeWitte-2k}
	$\Cayd(\integer_{2k} ; a, b, b+k)$, for $a,b,k \in \integer^+$, such that certain technical conditions \pref{2kConds} are satisfied.
	\end{enumerate}
\end{thm}

\begin{rem} \label{2kConds}
The precise conditions in \pref{LockeWitte-2k} are: (i)~either $a$ or~$k$ is odd, (ii)~either $a$~is even or $b$ and~$k$ are both even, (iii)~$\gcd(a-b,k) = 1$, (iv)~$\gcd(a,2k) \neq 1$, and (v)~$\gcd(b,k) \neq 1$.
\end{rem}

It is interesting to note that, in the examples provided by \cref{LockeWitte}, the group~$G$ is cyclic (either $\integer_{12k}$ or ~$\integer_{2k}$), and either
	\begin{enumerate}
	\item[\pref{LockeWitte-12k}] one of the generators has order~$2$,
	or
	\item[\pref{LockeWitte-2k}] two of the generators differ by an element of order~$2$.
	\end{enumerate}
S.\,J.\,Curran (personal communication) asked whether the constructions could be generalized by allowing $G$ to be an abelian group that is not cyclic. We provide a negative answer for case \pref{LockeWitte-2k}:

\begin{prop} \label{3genAbelMustBeCyclic}
Let $G$ be an abelian group\/ {\upshape(}\!written additively{\upshape)}, and let $a,b,k \in G$, such that $k$ is an element of order\/~$2$. {\upshape(}Also assume $\{a,b,b+k\}$ consists of three distinct, nontrivial elements of~$G$.{\upshape)} If the Cayley digraph $\Cayd(G;a,b,b+k)$ is connected, but does not have a hamiltonian cycle, then $G$ is cyclic.
\end{prop}

\begin{proof}
We prove the contrapositive: assume $G$ is not cyclic, and we will show that the Cayley digraph has a hamiltonian cycle (if it is connected). The argument is a modification of the proof of \cite[Thm.~4.1($\Leftarrow$)]{LockeWitte}.\refnote{3genAbelMustBeCyclicAid}

Construct a subdigraph~$H_0$ of~$G$ as in \cite[Defn.~4.2]{LockeWitte}, but with $G$ in the place of~$\integer_{2k}$, with $|G|$ in the place of~$2k$, and with $|a|$ in the place of~$d$. (Case~1 is when $k \notin \langle a \rangle$; Case~2 is when $k \in \langle a \rangle$.) Every vertex of~$H_0$ has both invalence~$1$ and outvalence~$1$.

The argument in Case~3 of the proof of \cite[Thm.~4.1($\Leftarrow$)]{LockeWitte} shows that $\Cayd(G; a,b,b+k)$ has a hamiltonian cycle if $\langle a - b, k \rangle \neq G$. Therefore, we may assume $\langle a - b, k \rangle = G$. On the other hand, we know $\langle a-b \rangle \neq G$ (because $G$ is not cyclic). Since $|k| = 2$, this implies $G = \langle a-b \rangle \oplus \langle k \rangle$.
Since $G$ is not cyclic, this implies that $a-b$ has even order. Also, we may write $a = a' + k'$ and $b = b' + k''$ for some (unique) $a',b' \in \langle a - b \rangle$ and $k', k'' \in \langle k \rangle$. (Since $a' - b' \in \langle a - b \rangle$, it is easy to see that $k' = k''$, but we do not need this fact.)

\begin{claim}
$H_0$ has an odd number of connected components.
\end{claim}
Arguing as in the proof of \cite[Lem.~4.1]{LockeWitte} (except that, as before, Case~1 is when $k \notin \langle a \rangle$, and Case~2 is when $k \in \langle a \rangle$), we see that the number of connected components in~$H_0$ is
$$ \begin{cases}
 |G : \langle a, k \rangle| +  |G : \langle b, k \rangle|
& \text{if $k \notin \langle a \rangle$}  , \\
 |G : \langle b, k \rangle|
& \text{if $k \in \langle a \rangle$} .
\end{cases}$$
Since $\langle a' - b' \rangle = \langle a - b \rangle$, we know that one of~$a'$ and~$b'$ is an even multiple of $a - b$, and the other is an odd multiple. (Otherwise, the difference would be an even multiple of $a - b$, so it would not generate $\langle a - b \rangle$.)
 Thus, one of $|G : \langle a, k \rangle|$ and $|G : \langle b, k \rangle|$ is even, and the other is odd. So $|G : \langle a, k \rangle| + |G : \langle b, k \rangle|$ is odd. This establishes the claim if  $k \notin \langle a \rangle$.
 
We may now assume $k \in \langle a \rangle$.
This implies that the element $a'$ has odd order (and $k'$ must be nontrivial, but we do not need this fact). This means that $a'$ is an even multiple of $a-b$, so $b'$~must be an odd multiple of $a-b$ (since $\langle a' - b' \rangle = \langle a - b \rangle$). Therefore $|\langle a - b \rangle : \langle b' \rangle|$ is odd, which means $|G : \langle b,k \rangle|$ is odd. This completes the proof of the claim.
\medbreak 

Now, if $|G : \langle b, k \rangle|$ is odd, we can apply a very slight modification of the argument in Case~4 of the proof of \cite[Thm.~4.1($\Leftarrow$)]{LockeWitte}. (Subcase~4.1 is when $k \notin \langle a \rangle$ and Subcase~4.2 is when $k \in \langle a \rangle$.) We conclude that $\Cayd(G;a,b,b+k)$ has a hamiltonian cycle, as desired.

Finally, if $|G : \langle b, k \rangle|$ is even, then more substantial modifications to the argument in \cite{LockeWitte} are required. For convenience, let $m = |G : \langle a, k \rangle|$. Note that, since $|G : \langle b, k \rangle|$ is even, the proof of the claim 
shows that $m$ is odd and $k \notin \langle a \rangle$.

Define $H_0'$ as in Subcase~4.1 of  \cite[Thm.~4.1($\Leftarrow$)]{LockeWitte} (with $G$ in the place of~$\integer_{2k}$, and replacing $\gcd(b,k)$ with $|G : \langle b, k \rangle|$). Let $H_1 = H_0'$, and inductively construct, for $1 \le i \le (m+1)/2$, an element $H_i$ of $\mathcal{E}$, such that
	$$ \{\, v \mid \text{$z_v = 0$ and $0 \le y_v \le 2i-2$} \,\}
	\cup
	\{\, v \mid \text{$z_v = 1$ and $x_v = 0$ or~$1$ (mod $|G : \langle b, k \rangle|$)} \,\} $$
is a component of~$H_i$, and all other components are components of~$H_0$. The construction of~$H_i$ from~$H_{i-1}$ is the same as in Subcase~4.1, but with $2i$ replaced by~$2i-1$.

We now let $K_1 = H_{(m+1)/2}$, and inductively construct, for $1 \le i \le |G : \langle b, k \rangle|/2$, an element~$K_i$ of~$\mathcal{E}$, such that
	$$ \{\, v \mid z_v = 0 \,\}
	\cup
	\{\, v \mid \text{$z_v = 1$ and $x_v \equiv 0,1,\ldots,$ or $2i-1$ (mod $|G : \langle b, k \rangle|$)} \,\} $$
is a single component of~$K_i$. 
Namely, \cite[Lem.~4.2]{LockeWitte} implies there is an element $K_i = K_{i-1}'$, such that $(2i-2)a$, $(2i-2)a + k$, and $(2i-1)a + k$ are all in the same component of~$K_i$. Then, for  $i = |G : \langle b, k \rangle|/2$, we see that $K_i$ is a hamiltonian cycle.
\end{proof}

\begin{ack}
I thank Stephen J.~Curran for asking the question that inspired \cref{3genAbelMustBeCyclic}.
The other results in this paper were obtained during a visit to the School of Mathematics and Statistics at the University of Western Australia (partially supported by funds from Australian Research Council Federation Fellowship FF0770915). I am grateful to my colleagues there for making my visit so productive and enjoyable.
\end{ack}


\newpage

\appendix

\pagestyle{myheadings}
\markboth{Notes to aid the referee (\emph{On Cayley digraphs that do not have hamiltonian paths} by Dave Witte Morris)}{Notes to aid the referee (\emph{On Cayley digraphs that do not have hamiltonian paths} by Dave Witte Morris)}

\section{Notes to aid the referee}

\thispagestyle{empty}

\bigskip\hrule width\textwidth \bigbreak

\begin{aid} \label{S=ab} \tolerance=5000
When $S = \{s_1,s_2,\ldots,s_n\}$, we often suppress the set braces, and write $\Cayd(G; s_1,s_2,\ldots,s_n)$, instead of $\Cayd \bigl( G; \{s_1,s_2,\ldots,s_n \} \bigr)$.
\end{aid}

\begin{aid} \label{ProveMilnor}
We provide a proof, since the result is stated without proof in \cite[p.~201]{Nathanson-PartProds} (and \cite[p.~267]{HolsztynskiStrube}). We use the notation and terminology of \cref{PrelimSect}. 

Suppose $L = (s_i)_{i=1}^n$ is a hamiltonian path. 
Since $|a| = 2$ and $|b| = 3$, we know that $(s_i)_{i=1}^n$ cannot contain two consecutive $a$'s, or three consecutive~$b$'s. On the other hand, if $[g](a,b,a)$ is a subpath of~$L$, then $gb^2$ must be either the initial vertex or the terminal vertex of~$L$. (Since $gb^2 \cdot b = g$ and $gb^2 \cdot b^{-1} = gb$ are interior vertices of the path $[g](a,b,a)$, we see that $gb^2 a$ is the only vertex of~$G$ that can be adjacent to~$gb^2 a$ in~$L$.) Therefore, there can be at most two occurrences of $(a,b,a)$ in $(s_i)_{i=1}^n$. (And there must be less than two occurrences unless $s_1 = s_n = a$.) Hence, no path can be longer than
	$$ \bigl( (a, b^2)^{|ab^2|-1}, a,b, (a, b^2)^{|ab^2|-1}, a,b, (a, b^2)^{|ab^2|-1}, a \bigr) ,$$
which has length $9 |ab^2| - 4$. Therefore $|G| \le 9 |ab^2| - 3$. This is a slightly better bound than is stated in the \lcnamecref{23NoHP}.
\end{aid}

\begin{aid} \label{TermIndep}
For $a,b \in S$, we have
	$ a^{-1} H = b^{-1} (ba^{-1}) H = b^{-1} H$,
since $ba^{-1} \in S S^{-1} \subseteq H$.
\end{aid}

\begin{aid} \label{Sa}
For $s,t \in S$, we have 
	$$s t^{-1} = (s g) (g^{-1} t^{-1}) = (sg) (tg)^{-1} \in (Sg)(Sg)^{-1} \subseteq \langle Sg \rangle .$$
Therefore $\langle S S^{-1} \rangle \subseteq \langle Sg \rangle$.

For $a \in S$, we obviously have $\langle S a^{-1} \rangle \subseteq \langle S S^{-1} \rangle$. Letting $g = a^{-1}$ in the conclusion of the preceding paragraph provides the opposite inclusion.
\end{aid}

\begin{aid} \label{aS}
Essentially the same argument as in \pref{Sa} shows $\langle S^{-1} S \rangle = \langle a^{-1} S \rangle$.

We have $a^{-1} S = a^{-1} (Sa^{-1}) a = (Sa^{-1})^a$, so $\langle a^{-1} S \rangle = \langle Sa^{-1} \rangle^a$.

From \fullcref{ArcForcingRems}{Sg}, we know $\langle Sa^{-1} \rangle = \langle S S^{-1} \rangle$.
\end{aid}

\begin{aid} \label{HousmanHAid}
Letting $S = \{a,b\}$, the arc-forcing subgroup is 
	$$H = \langle S b^{-1} \rangle
	 = \bigl\langle \{a,b\} b^{-1} \bigr\rangle
	 = \langle ab^{-1}, b b^{-1}  \rangle
	 = \langle ab^{-1}, e  \rangle
	 = \langle ab^{-1}  \rangle
	 . $$
\end{aid}

\begin{aid} \label{HousmanAid}
The proof of the \lcnamecref{HousmanThm} is so short that we provide it for the reader's convenience. The idea goes back to \cite{Rankin-Camp}.

Note that if $g$ travels by~$b$, then $g(ba^{-1})$ cannot travel by~$a$. (Otherwise, $L$~would visit the vertex $gb = \bigl(g(ba^{-1}) \bigr) a$ twice. Thus, either $g(ba^{-1})$ travels by~$b$, or $g(ba^{-1})$ is the terminal vertex. Hence, if $gH$ does not contain the terminal vertex, then we see, by induction, that $g(ba^{-1})^k$ travels by~$b$ for all $k \in \integer$. So $gH$ travels by~$b$.

Similarly, if $g$ travels by~$a$, then $g(ab^{-1})$ cannot travel by~$b$. Thus, $gH$ travels by~$a$, unless $gH$ contains the terminal vertex.

Therefore, a coset $gH$ either travels by~$a$ or travels by~$b$, unless it contains the terminal vertex.

Furthermore, since no directed edge of~$L$ can enter the initial vertex~$e$, we know that $a^{-1}$ does not travel by~$a$ and $b^{-1}$~does not travel by~$b$. So $a^{-1}H$ does not travel by~$a$, and $b^{-1} H$ does not travel by~$b$. However, $a^{-1}H = b^{-1} H$, since $(b^{-1})^{-1} a^{-1} = b a^{-1} \in H$. Therefore, $a^{-1}H$ travels by neither $a$ nor~$b$. This proves \pref{HousmanThm-terminal}.

We now know that no regular coset contains the terminal vertex. Therefore, each regular coset either travels by~$a$ or travels by~$b$. This proves \pref{HousmanThm-regular}.

\medbreak

For future reference, we record the following observation that follows from the above arguments:

\begin{lem}[Housman {\cite[Lem.~6.4(b)]{CurranWitte}}] \label{TermCosetTravels}
Assume the situation of \cref{HousmanThm}. Any element of the terminal coset is of the form $a^{-1} (b a^{-1})^i$, with $0 \le i < |a b^{-1}|$. In particular, from \fullcref{HousmanThm}{terminal}, we know the terminal vertex of~$L$ is $a^{-1} (b a^{-1})^d$, for some~$d$ with $0 \le d < |a b^{-1}|$. Then:
	$$ a^{-1} (b a^{-1})^i \text{ travels by }
	\begin{cases}
	 b & \text{if $0 \le i < d$} , \\
	 a & \text{if $d < i <  |a b^{-1}|$} 
	 . \end{cases} $$
Therefore, the number vertices in the terminal coset that travel by~$b$ is exactly~$d$.
\end{lem}

\begin{proof}
Since no edge of~$L$ enters the initial vertex~$e$, we know that $b^{-1}$ does not travel by~$b$. This is the base case of a proof by induction that $a^{-1} (b a^{-1})^i$ travels by~$a$ if $0 \le i < d$. The induction step is provided by the argument in the second paragraph of \pref{HousmanAid}. After interchanging $a$ and~$b$, the same argument shows that $b^{-1} (a b^{-1})^j$ travels by~$b$ if $0 \le j <  |a b^{-1}| - d - 1$.
\end{proof}

\end{aid}

\begin{aid} \label{SetupH}
We have $\langle a b^{-1} \rangle 
= \langle (\ah z)(\bh z)^{-1} \rangle
= \langle (\ah z)(z^{-1} \bh^{-1}) \rangle
= \langle \ah \bh^{-1} \rangle$.

By the definition of~$\beta$, we have $\gcd(\alpha,\beta) = 1$. Therefore $\integer_\alpha \times \integer_\beta$ is cyclic. More precisely, since $\ah$~generates $\integer_\alpha$, and $\bh^{-1}$~generates $\integer_\beta$, we have $\langle \ah \bh^{-1} \rangle = \integer_\alpha \times \integer_\beta$.
\end{aid}

\begin{aid} \label{verttrans}
For each $g \in G$, define $\varphi_g \colon G \to G$ by $\varphi_g(x) = gx$. It is easy to see that the map  is an automorphism of $\Cayd(G;S)$. Namely, if there is an edge from~$x$ to~$y$, then we have $y = xs$ for some $s \in S$. Then
	$$ \varphi_g(x) \, s = (gx) s= g(xs) = \varphi_g(xs)  =  \varphi_g(y),$$
so there is an edge from $\varphi_g(x)$ to $\varphi_g(y)$. 

Also, for any $x, y \in G$, we have $\varphi_{yx^{-1}}(x) = (yx^{-1})x = y$. Therefore the group $\{\, \varphi_g \}$ of automorphisms of $\Cayd(G;S)$ acts transitively on the set of vertices, so (by definition) $\Cayd(G;S)$ is vertex-transitive.

Now, if $g$ is the initial vertex of~$L$, then $e$~is the initial vertex of the hamiltonian path $\varphi_{g^{-1}}(L)$.
\end{aid}

\begin{aid} \label{TermCosetz/MinOne}
$a^{-1} H = (\ah z)^{-1} H = z^{-1} \ah^{-1} H = z^{-1} H$, since $\ah \in \integer_\alpha \times \integer_\beta = H$.

The element $r^2$ generates a subgroup of index~$2$ (and order $(p-1)/2$) in~$\integer_p^\times$. Since $p \equiv 3 \pmod{4}$, we know that $(p-1)/2$ is odd, so $\langle r^2 \rangle$ does not contain any elements of even order. In particular, it does not contain~$-1$, which is of order~$2$. Therefore $\langle -1, r^2 \rangle$ properly contains $\langle r^2 \rangle$. Since $|\integer_p^\times : \langle r^2 \rangle| = 2$, this implies $\langle -1, r^2 \rangle = \integer_p^\times$.
\end{aid}

\begin{aid} \label{AllzH}
We have
	$ G = (\integer_\alpha \times \integer_\beta) \ltimes \integer_p
	= H \integer_p
	= \integer_p H
	$
(since $\integer_p \normal G$). Therefore, every left coset of~$H$ is of the form $z' H$, for some $z' \in \integer_p$.
\end{aid}

\begin{aid} \label{CentralIsNilp}
By definition \cite[Defn.~5.7.8, p.~23]{Ash-AbstractAlgebra}, $G$ is \emph{nilpotent} if there exists a chain
	$$ \{e\} = Z_0 \normal Z_1 \normal \cdots \normal Z_c = G ,$$
of subgroups of~$G$, such that $[G, Z_i] \subseteq Z_{i-1}$ for $i \ge 1$. If $[G, G] \subseteq Z(G)$, where $Z(G)$ is the center of~$G$, then the chain
	$$ \{e\} \normal [G,G] \normal G ,$$
shows that $G$ is nilpotent, because $\bigl[ G, [G,G] \bigr] \subseteq \bigl[G, Z(G) \bigr] = \{e\}$.
\end{aid}

\begin{aid} \label{MayAssumeS=2Aid}
Let $\pi = t_1 t_2 \cdots t_n$. 	
\begin{itemize}
\item The path $(t_j)_{j=1}^n$ traverses the vertices in~$N$. 
\item Then $[\pi s_1](t_j)_{j=1}^n$ traverses the vertices in $\pi s_1 N = N s_1$.
\item Then $[\pi s_1\pi s_2](t_j)_{j=1}^n$ traverses the vertices in $\pi s_1 \pi s_2 N = N s_1 s_2$.
	\\ \hbox{\qquad} \dots
\item Then $[\pi s_1\pi s_2 \cdots \pi s_m](t_j)_{j=1}^n$ traverses the vertices in $\pi s_1 \pi s_2 \cdots \pi s_m N = N s_1 s_2 \cdots s_m$.
\end{itemize}
Since $(s_i)_{i=1}^m$ is a hamiltonian path in $\Cayd(G/N;S)$, we know that $N, N s_1, \ldots N s_1 s_2 \cdots s_m$ is a list of all the cosets of~$N$, so the walk traverses all the vertices in~$G$ (without repetition), and is therefore a hamiltonian path.
\end{aid}

\begin{aid} \label{SkewedGensAid}
Since $(s_i)_{i=1}^{n-1}$ is a hamiltonian path in $K \backslash {\Cayd(G;S)}$, any $g \in G$ can be written (uniquely) in the form 
	$$ \text{$g = k s_1 s_2 \cdots s_p$, with $k \in K$ and $1 \le p < n$} .$$
Let $t_j' = t_j s_2 s_3 \cdots s_n$ for $1 \le j \le m$. Then $(t_j')_{j=1}^m$ is a hamiltonian path in $\Cayd(K; S s_2 s_3 \cdots s_n)$, so there is a (unique)~$q$, such that $t_1' t_2' \cdots t_q' = k$ (and $1 \le q \le m$). Hence, $g$~can be written uniquely in the form
	$$ t_1' t_2' \cdots t_q' \cdot s_1 s_2 \cdots s_p .$$
This means that the walk visits each vertex~$g$ exactly once, so it is a hamiltonian path.
\end{aid}

\begin{aid} \label{GGPrimeHGan}
Note that $a$~generates the quotient group $G/H^G$, since
	$$G = \langle S \rangle \subseteq \langle S a^{-1}, a \rangle = \langle H, a \rangle \subseteq \langle H^G, a \rangle .$$
Therefore $(a^n)$ is a hamiltonian cycle in $\Cayd(G/H^G; S)$.

Also, we have $a \equiv b \equiv c \pmod{H^G}$, since 
	$$\text{$a b^{-1} \in S S^{-1} \subseteq H \subseteq H^G$ and 
	$a c^{-1} \in S S^{-1} \subseteq H \subseteq H^G$} .$$
Therefore, replacing some or all of the occurrences of~$a$ in $(a^n)$ with either $b$ or~$c$ will have no effect on the hamiltonian cycle in $\Cayd(G/H^G; S)$. 
In other words, if $s_i \in \{a,b,c\}$ for all~$i$, then $(s_i)_{i=1}^n$ is a hamiltonian cycle in $\Cayd(G/H^G; S)$.
In particular, $(a^{n-1}, b)$ and $(a^{n-2}, b, c)$ are hamiltonian cycles.
\end{aid}

\begin{aid} \label{GGPrimeHGM}
Let $H' = \langle H, s_1s_2 \cdots s_n \rangle$, so $H'$ is a subgroup of~$H^G$. 

Suppose $H'$ is a proper subgroup of~$H^G$. Then $H'$ is contained in some maximal subgroup~$M'$ of~$H^G$. Then, since $H'$ obviously contains~$H$, we we see that $M'$ contains~$H$, so the uniqueness of~$M$ implies $M' = M$. Therefore 
	$$ s_1s_2 \cdots s_n \in H' \subseteq M' = M .$$
This contradicts the fact that $s_1s_2 \cdots s_n \notin M$. We conclude that $H' = H^G$, which establishes the first equality.

\medbreak

\fullCref{ArcForcingRems}{Sg} tells us $H \subseteq \langle S s_2 \cdots s_n \rangle$. 
Since $s_1 \in S$, we also have $s_1 s_2 \cdots s_n \in S s_2 \cdots s_n$.
Therefore $\langle H, s_1 s_2 \cdots s_n \rangle \subseteq \langle S s_2 \cdots s_n \rangle$.
For the reverse inclusion, note that
	$$ S s_2 \cdots s_n = (S s_1^{-1}) (s_1 s_2 \cdots s_n) \subseteq Hs_1 s_2 \cdots s_n \subseteq \langle H , s_1 s_2 \cdots s_n \rangle .$$
\end{aid}

\begin{aid} \label{UniqueMaxClosure}
For any $h \in G$, we have
	$$ (H^G)^h 
	= \langle\, H^g \mid g \in G \,\rangle^h
	= \langle\, H^{gh} \mid g \in G \,\rangle
	= \langle\, H^x \mid x \in Gg \,\rangle
	= \langle\, H^x \mid x \in G \,\rangle
	= H^G
	,$$
so $H^G \normal G$.

Now, suppose $N$ is any normal subgroup that contains~$H$. Then, for every $g \in G$, we have $H^g \subseteq N^g = N$. Therefore $H^G = \langle\, H^g \mid g \in G \,\rangle 
\subseteq N$.

So $H^G$ is indeed the unique smallest normal subgroup of~$G$ that contains~$H$. This is well known.

\medbreak

It is also well known that if $K$ is any subgroup of~$G$ that contains $[G,G]$, then $K \normal G$. (So, in particular, $H \, [G,G] \normal G$.) To see this, note that if $k \in K$ and $g \in G$, then
	$$ k^g = g^{-1} k g = k (k^{-1} g^{-1} k g ) = k [k,g] \in K ,$$
so $K^g \subseteq K$.
\end{aid}

\begin{aid} \label{UniqueMaxHK}
The first equality is a special case of the well known fact that if $H$, $K$, and~$M$ are subgroups of~$G$, such that $HK = L$ and $H \subseteq M \subseteq HK$, then $M = H \cdot (M \cap K)$. To prove this fact, first note that the inclusion ($\supseteq$) is obvious, since $H$ and $M \cap K$ are contained in~$M$. Given $m \in M$, we know, by assumption, that $m \in HK$, so may write $m = h k$ with $h \in H$ and $k \in k$. Then $k = h^{-1} m \in M$, since $h \in H \supseteq M$ and $m \in M$. Therefore $k \in K \cap M$. So $m = hk \in H\cdot (M \cap K)$, as desired.

For the second inequality, notice that if $\langle x \rangle$ is a nontrivial cyclic group of order~$p^\ell$, then every proper subgroup of $\langle x \rangle$ is contained in~$\langle x^p \rangle$. Now simply let $\langle x \rangle = H^G \cap [G,G]$.
\end{aid}

\begin{aid} \label{OnlyInvertHPGG}
It is well known that if $[G,G]$ is cyclic (and $S$ is a generating set of~$G$, as usual), then
	$$ [G,G] = \bigl\langle [a,b]  \mid a,b \in S \bigr\rangle .$$
For the reader's convenience, we reproduce the proof of this that appears in \cite[Lem.~3.5]{GhaderpourMorris-Nilpotent}. Let $N = \bigl\langle [a,b]  \mid a,b \in S \bigr\rangle$. Then $N \normal G$, because $N$ is contained in the cyclic, normal subgroup $[G,G]$, and every subgroup of a cyclic, normal subgroup is normal. In $G/N$, every element of~$S$ commutes with all of the other elements of~$S$, so $G/N$ is abelian. Hence $[G,G] \subseteq N$.

In our case, we have an isomorphism $\varphi \colon [G,G] \to \integer_{p^k}$. Since the multiples of~$p$ form a proper subgroup of~$\integer_{p^k}$, we know, from the preceding paragraph, that there exist $a,b \in [G,G]$, such that $\varphi \bigl( [a,b] \bigr)$ is not a multiple of~$p$. Then $\gcd \bigl( \varphi \bigl( [a,b] \bigr), p^k \bigr) = 1$, so $\bigl\langle \varphi \bigl( [a,b] \bigr) \bigr\rangle = \integer_{p^k}$. Since $\varphi$ is an isomorphism, this tells us $\langle [a,b] \rangle = [G,G]$.
\end{aid}

\begin{aid} \label{OnlyInvertHPHz}
We have 
	$$ G = \langle a,b \rangle \subseteq \langle H z, Hz \rangle \subseteq \langle H, z \rangle = H \langle z \rangle ,$$
since $\langle z \rangle \normal G$. (Recall that every subgroup of a cyclic, normal subgroup is normal.)
\end{aid}

\begin{aid} \label{OnlyInvertHPan}
Note that $\ah = a z^{-1}$ must centralize~$\integer_{p^k}$, since $a$ and~$z$ both centralize it. Hence, for any~$k$, we have
	$$ H a^k = H (\ah z)^k = H \ah^k z^k = H z^k ,$$
since $\ah \in H$.

Also, since $H \langle z \rangle = G$, it is obvious that $[H, Hz, Hz^2,\ldots,Hz^{n-1}, H]$ is a hamiltonian cycle. 
\end{aid}

\begin{aid} \label{A4xZ2NoPath}
An exhaustive search will quickly show there is no hamiltonian path (see \cite[p.~266]{HolsztynskiStrube} for a picture of the digraph), but we use some theory instead of case-by-case analysis.

Suppose $L = (s_i)_{i=1}^{23}$ is a hamiltonian path in $\Cayd(G; a,b)$. Each left coset $g \langle b \rangle$ of $\langle b \rangle$ cannot contain more than two $b$-edges (since $[g](b^3)$ is a cycle). Since there are $8$ such cosets, this means that $L$ cannot have more than $16$ $b$-edges.  In fact, there must be strictly less than $16$, because otherwise $L$ would contain the cycle $(b^2, a)^6$.

On the other hand, the argument in \pref{ProveMilnor} tells that if a left coset $g \langle b \rangle$ does not contain either the initial vertex or the terminal vertex, then it does contain two $b$-edges. Thus, there cannot be more than two cosets that do not have exactly two $b$-edges. Furthermore, the same line of reasoning shows that each coset must have at least one $b$-edge. So $L$ has at least $16 - 2 = 14$ $b$-edges.

In summary, the number of $b$-edges is either $14$ or~$15$. Therefore, exactly two regular cosets travels by~$b$, and $2$ or~$3$ vertices in the terminal coset travel by~$b$. 

Assume $e$ is the initial vertex of~$L$, so any vertex in the terminal coset can be written in the form $a^{-1} (b a^{-1})^i = b^{-1} (a b^{-1})^{5-i}$, with $0 \le i \le 5$. Let $d$ be the number of vertices that travel by~$b$ in the special coset. 
Then \cref{TermCosetTravels} tells us:
	\begin{itemize}
	\item $a (b a)^i$ travels by~$b$ iff $0 \le i < d$, 
	\item $b^{-1} (a b^{-1})^j$ travels by~$a$ iff $0 \le j \le 5 - d$, 
	and
	\item $a(b a)^d$ is  the terminal vertex.
	\end{itemize}
Since $5 - d \ge 2$, we know that $b^{-1}$ and $b^{-1}(a b^{-1})$ both travel by~$a$. So $b^{-1} a$ and $(b^{-1} a)^2 = \bigl( b^{-1}(a b^{-1}) \bigr) a$ cannot travel by~$a$. (Otherwise, $L$ would contain a cycle of the form $[g](a^2)$. So the two regular cosets $(b^{-1} a) \langle ba \rangle$ and $(b^{-1}a)^2 \langle ba \rangle$ must travel by~$b$. 

Letting $z = (e,1)$ be the nonidentity element of $\{e\} \times \integer_2$, we have $(ab^{-1})^3 = (ab)^3 = z$, so $(ab^{-1})^3 (ab)^3 = e$. Therefore, for $g = (ab)^2$, we have
	\begin{align*}
	g &= (ab)^2 = [(ab^{-1})^3(ab)]^{-1} = (b^{-1}a)(ba)^3 \in (b^{-1}a) \langle ba \rangle , \\
	gb &= (ab)^2 b = (ab)(ab^{-1}) = [(ab^{-1})^2 (ab)^2]^{-1} = (b^{-1} a)^2 (ba)^2  \in (b^{-1}a)^2 \langle ba \rangle , \\
	gb^2 &= (ab)^2 b^{-1} = a(ba) = a(ba)^i \text{ with $i = 1 < d$}
	,\end{align*}
so all three of these vertices travel by~$b$. This means that $L$ contains the cycle $[g](b^3)$, which contradicts the fact that $L$ is a (hamiltonian) path.
\end{aid}

\begin{aid} \label{abelianpathaid}
We reproduce the proof, since it is so short.

Let $S_0 = S \smallsetminus \{s\}$, for some $s \in S$. By induction on $\#S$, we may assume there is a hamiltonian path $(t_i)_{i=1}^m$ in $\Cayd \bigl( G / \langle s \rangle ; S_0 \bigr)$. Then
	$ \bigl( (s^{|s|-1}, t_i)_{i=1}^m, s^{|s|-1} \bigr) $
is a hamiltonian path in $\Cayd(G;S)$.
\end{aid}

\begin{aid} \label{RankinAid}
We sketch a short proof, since the argument in \cite[Thm.~4]{Rankin-Camp} is lengthy. We begin with a well-known, elementary observation.

\begin{lem*}[``Factor Group Lemma'' {\cite[\S2.2]{WitteGallian-survey}}]
Suppose
	\begin{itemize}
	\item $N$ is a cyclic, normal subgroup of~$G$,
	and
	\item $(s_i)_{i=1}^d$ is a hamiltonian cycle in $\Cayd(G/N; S)$.
	\end{itemize}
Then $\bigl( (s_i)_{i=1}^d \bigr)^{|N|}$ is a hamiltonian cycle in $\Cayd(G;S)$ if and only if $\langle s_1 s_2 \cdots s_d \rangle = N$.
\end{lem*}

\begin{proof}
Since $(s_i)_{i=1}^{d-1}$ is a hamiltonian path in $\Cayd(G/N; S)$, we know that every element of~$G$ can be written uniquely in the form $x s_1 s_2 \cdots s_j$, with $x \in N$ and $0 \le j < d$. Therefore, we have the following equivalences:
	\begin{align*}
	&\langle s_1 s_2 \cdots s_d \rangle = N \\
	&\iff  \text{every element  of~$G$ can be written uniquely in the form $(s_1 s_2 \cdots s_d)^i s_1 s_2 \cdots s_j$,} \\
	& \hskip 1in \text{with $0 \le i < |N|$ and $0 \le j < d$} \\
	 &\iff \text{$\bigl( (s_i)_{i=1}^d \bigr)^{|N|}$ is a hamiltonian cycle}
	. \qedhere \end{align*} 
\end{proof}

\begin{proof}[\bf Proof of \cref{Rankin2gen}]
Let $d = |G : \langle a b^{-1} \rangle|$.

($\Rightarrow$) Suppose $C$ is a hamiltonian cycle in $\Cayd(G; a,b)$. Then \fullcref{HousmanThm}{regular} tells us that each left coset of $\langle a b^{-1} \rangle$ either travels by~$a$ or travels by~$b$ (since the cycle $C$ has no terminal vertex). Therefore $C = \bigl( (s_i)_{i=1}^d \bigr)^{|a b^{-1}|}$, for some $s_1,\ldots,s_d \in \{a,b\}$. Let $k$ (resp.~$\ell$) be the number of cosets that travel by~$a$ (resp.~$b$), so $k + \ell = d$, and $s_1 s_2 \cdots s_d = a^k b^\ell$ (since $G$ is abelian). The Factor Group Lemma tells us that $\langle s_1 s_2 \cdots s_d \rangle = \langle a b^{-1} \rangle$.

($\Leftarrow$) Since $(a^k, b^\ell)$ is a hamiltonian cycle in $\Cayd \bigl( G/ \langle a b^{-1} \rangle ; a, b \bigr)$, and $\langle a^k b^\ell \rangle = \langle a b^{-1} \rangle$, the Factor Group Lemma tells us that $(a^k, b^\ell)^d$ is a hamiltonian cycle.
\end{proof}
\end{aid}

\begin{aid} \label{EasyNoHam}
Let $G = \integer_n$ and $b = a + 1$. Then 
	$$\langle a - b \rangle = \langle a - (a + 1) \rangle = \langle -1 \rangle = \integer_n ,$$
so $|G : \langle a - b \rangle| = 1$. Therefore, if $\Cayd(G;a,b)$ has a hamiltonian cycle, then \cref{Rankin2gen} tells us there exist $k, \ell > 0$, such that $k + \ell = 1$ and $\gcd(ka + \ell b, n) = 1$. However, since $k + \ell = 1$, the sum $ka + \ell b$ must simply be either~$a$ or~$b$. By assumption, neither of these is relatively prime to~$n$. This is a contradiction.
\end{aid}

\begin{aid} \label{3genAbelMustBeCyclicAid}
See \cref{3genAbelMustBeCyclicSect} for an expanded proof of \cref{3genAbelMustBeCyclic} that includes appropriately modified excerpts from~\cite{LockeWitte}. 
\end{aid}

\newpage

\section{\texorpdfstring{Proof of \cref{3genAbelMustBeCyclic}}{A proof for abelian groups}} 
\label{3genAbelMustBeCyclicSect}

\vskip-0.5\baselineskip

\centerline{(adapted from the proof of \cite[Thm.~4.1($\Leftarrow$)]{LockeWitte})}

\medskip

\markboth{Proof of \cref{3genAbelMustBeCyclic} (\emph{On Cayley digraphs that do not have hamiltonian paths} by Dave Witte Morris)}{Proof of \cref{3genAbelMustBeCyclic} (\emph{On Cayley digraphs that do not have hamiltonian paths} by Dave Witte Morris)}

\thispagestyle{empty}

\newcommand{\qf}{\mathord{\mathcal E}}
\newcommand{\qc}{\mathord{\mathcal C}}

Let 
	\begin{itemize}
	\item $G$ be an abelian group (written additively), 
	and 
	\item $a,b,k \in G$, such that $k$ is an element of order~$2$. 
	\end{itemize}
Assume
	\begin{itemize}
	\item $\{a,b,b+k\}$ consists of three distinct, nontrivial elements of~$G$,
	\item $G$ is \underline{not} cyclic, 
	and 
	\item $\langle a, b, k \rangle = G$ (or, equivalently, $\Cayd(G;a,b,b+k)$ is connected).
	\end{itemize}
We will show that $\Cayd(G;a,b,b+k)$ has a hamiltonian cycle.

\medbreak

\begin{defn}
Let 
	\begin{itemize}
	\item $\qc$ be the set of all
spanning subdigraphs of $\Cayd(G;a,b,b+k)$ with invalence~$1$ and  outvalence~$1$
at each vertex (so each connected component is a directed cycle),
	and
	\item $\qf$ be the set of all elements of~$\qc$, such that, in each coset of
the subgroup $\{0,k\}$, one vertex travels by~$a$, and
the other vertex travels by either $b$~or~$b+k$. 
	\end{itemize}
 \end{defn}

\begin{defn} \label{H0-defn}
We construct an element $H_0$ of~$\qf$. 
The construction considers two cases.

\setcounter{case}{0}

\begin{case}
Assume $k \notin \langle a \rangle$.
 \end{case}
 In this case, every vertex~$v$ in~$G$ can be uniquely written in the form $x_v a + y_v
b + z_v k$ with $0 \le x_v < |a|$, $0\le y_v < |G : \langle a, k \rangle|$, and $0 \le
z_v < 2$.
 Let $H_0$ be the spanning subdigraph in which a vertex~$v \in G$ 
 \begin{itemize}
 \item travels by~$a$ if~$z_v = 0$;
 \item travels by~$b$ if $z_v = 1$ and $z_{v+b} = 1$; and
 \item travels by~$b+k$ otherwise.
 \end{itemize}
 (By construction, the vertices~$v$ that satisfy $z_v = 0$ are
both entered and exited via an $a$-arc in~$H_0$; the other
vertices are neither entered nor exited via an $a$-arc.)

\begin{case}
Assume $k \in \langle a \rangle$.
 \end{case}
 In this case, every vertex~$v$ in
$G$ can be uniquely written in the form $x_v a + y_v
b$ with $0 \le x_v < |a|$ and $0\le y_v < |G : \langle a \rangle|$. Let $H_0$ be the
spanning subdigraph in which a vertex~$v \in G$ 
 \begin{itemize}
 \item travels by~$a$ if~$x_v < |a|/2$;
 \item travels by~$b+k$ if $x_v \ge |a|/2$ and $1 \le x_{v+b} \le
|a|/2$; and
 \item travels by~$b$ otherwise.
 \end{itemize}
 (By construction, the vertices~$v$ that satisfy $1 \le x_v \le
|a|/2$ are precisely those that are entered via an $a$-arc
in~$H_0$.) 
 \end{defn}

\begin{lem}[{}{\cite[Lem.~2.1]{LockeWitte}}] \label{parity-lemma}
Suppose $H$~and~$H'$ belong to~$\qc$. Let
$u_1$, $u_2$, and~$u_3$ be three vertices of~$H$, and let $v_i$
be the vertex that follows~$u_i$ in~$H$. Assume  $H'$ has the
same arcs as~$H$, except:
 \begin{itemize}
 \item instead of the arcs from~$u_1$ to~$v_1$, from~$u_2$
to~$v_2$, and from~$u_3$ to~$v_3$, 
 \item there are arcs from~$u_1$ to~$v_2$, from~$u_2$ to~$v_3$,
and from~$u_3$ to~$v_1$.
 \end{itemize}
 Then the number of components of~$H$ has the same parity as the
number of components of~$H'$.
 \end{lem}

\begin{proof}  Let $\sigma$ be the permutation of~$\{1,2,3\}$
defined by: $u_{\sigma(i)}$ is the vertex that is encountered
when~$H$ first reenters $\{u_1,u_2,u_3\}$ after~$u_i$. Thus, if
$\sigma$ is the identity permutation, then $u_1,u_2,u_3$ lie on
three different components of~$H$. On the other hand, if
$\sigma$ is a $2$-cycle, then two of $u_1,u_2,u_3$ are on the
same component, but the third is on a different component.
Similarly, if $\sigma$ is a $3$-cycle, then all three of these
vertices are on the same component. Thus, the parity of the
number of components of~$H$ that intersect $\{u_1,u_2,u_3\}$ is
precisely the opposite of the parity of the
permutation~$\sigma$. 

There is a similar permutation~$\sigma'$ for~$H'$. From the
definition of~$H'$, we see that $\sigma'$ is simply the product
of~$\sigma$ with the $3$-cycle $(1,2,3)$, so $\sigma'$ has the
same parity as~$\sigma$, because $3$-cycles are even
permutations. Thus, the parity of the number of components
of~$H$ that intersect $\{u_1,u_2,u_3\}$ is the same as the
parity of the number of components of~$H'$ that intersect
$\{u_1,u_2,u_3\}$. Because the components that do not intersect
$\{u_1,u_2,u_3\}$ are exactly the same in~$H$ as in~$H'$, this
implies that the number of components in~$H$ has the same parity
as the number of components in~$H'$.
 \end{proof}

\begin{lem}[{}{\cite[Lem.~4.2]{LockeWitte}}] \label{amalgamate}
Assume $H \in \qf$, and
suppose $u$ is a vertex of~$H$ that travels by~$a$, such that
$u$, $u+k$, and $u + a + k$ are on three different components
of~$H$. Then there is an element~$H'$ of~$\qf$, with exactly the
same arcs as~$H$, except the arcs leaving~$u$ and~$u+k$, and the
arc entering~$u+a+k$, such that $u$, $u+k$, and~$u+a+k$ are all
on the same component of~$H'$.
 \end{lem}
 
 \begin{proof}
 Let
 	\begin{itemize}
	\item $u_1 = u$, 
	\item $u_2 = u + k$, 
	\item $v_3 = u + a + k$, 
	\item $u_3$ be the vertex that precedes $u_3$ on~$H$,
	and
	\item $v_1$ and~$v_2$ be the vertices that follow $u_1$ and~$u_2$, respectively, on~$H$.
	\end{itemize}
Note:
	\begin{itemize}
	\item Since $v_3 = u + a + k = u_2 + a$, there is an edge from $u_2$ to~$v_3$.
	\item Since $u_2 = u + k$ and $v_3 = u + a + k$ are not in the same component, we know that $u_2$ does not travel by~$a$. Therefore, it travels by either $b$ or~$b + k$, so $v_2 \in \{u + b, u + b + k\}$. Therefore, there is an edge from $u_1 = u$ to~$v_2$.
	\item Since $u + k = u_2$ does not travel by~$a$ (and $H$ is in $\qf$), we know that $u_1 = u$ travels by~$a$. So $v_1 = u_1 + a = u + a$. 
	\item Since $v_3 - a = u + k = u_2$ does not travel by~$a$, we know that $u_3$ travels by either $b$ or~$b + k$, so 
		$$u_3 \in \{v_3 - b, v_3 - (b+k)\} = \{u + a - b + k, u + a - b \} = \{v_1 - (b+k), v_1 - b\} ,$$
	so there is an edge from $u_3$ to~$v_1$.
	\end{itemize}
Hence, the proof of \cref{parity-lemma} provides us with the desired $H' \in \qf$.
 \end{proof}

The same argument yields the following:

\begin{lem}[{}{\cite[Lem.~4.2]{LockeWitte}}] \label{amalgamate23}
Assume $H \in \qf$, and
suppose $u$ is a vertex of~$H$ that travels by~$a$, such that
	\begin{itemize}
	\item $u+k$ and $u + a + k$ are in the same component of~$H$,
	but
	\item $u$ is in a different component.
	\end{itemize}
If $v$ is the vertex that immediately follows~$u+k$ in~$H$, then there is an element~$H'$ of~$\qf$, with exactly the
same arcs as~$H$, except the arcs leaving~$u$ and~$u+k$, and the
arc entering~$u+a+k$, such that 
	\begin{itemize}
	\item $u$ and~$v$ are in the same component of~$H'$,
	but
	\item $u + a + k$ is in a different component of~$H'$.
	\end{itemize}
 \end{lem}

\begin{lem} \label{IfabkNotG}
If $\langle a - b , k \rangle \neq G$, then $\Cayd(G; a, b, b+k)$ has a hamiltonian cycle.
\end{lem}

\begin{proof}
There are many spanning subdigraphs of $\Cayd(G; a, b, b+k)$ in which 
 \begin{itemize}
 \item every vertex has invalence~$1$ and outvalence~$1$,
 \item
 every vertex not in $\langle a-b,k \rangle$ travels by either
$b$~or~$b+k$, 
and
 \item for each vertex $v \in \langle a-b,k \rangle$, one of
$v$~and~$v+k$ travels by~$a$, and the other travels by either
$b$ or~$b+k$.
 \end{itemize}
 Among all such digraphs, let $H$ be one in which the number of
components is minimal.

We claim that $H$ is a hamiltonian cycle. If not, then $H$ has
more than one component. Because $\langle a,b,k \rangle =
G$, we know that $b$ generates the quotient group
$G/ \langle a-b,k \rangle$, so every component
of~$H$ intersects $\langle a-b,k \rangle$, and hence either 
 \begin{itemize}
 \item there is some vertex~$u$ in $\langle a-b,k \rangle$ such
that $u$~and~$u+k$ are in different components of~$H$; or 
 \item for all $v \in \langle a-b,k \rangle$, the vertices
$v$~and~$v+k$ are in the same component of~$H$, but there is
some vertex~$u$ in $\langle a-b,k \rangle$ such that $u$
and~$u+(a-b)$ are in different components of~$H$.
 \end{itemize}
 In either case, let $u_1$ be the one of $u$~and~$u+k$ that
travels by~$a$. 

Let $v_1 = u_1 + a$. Let $u_2 = u_1 + k$, and let $v_2 \in u_2 +
\{b,b+k\}$ be the vertex that follows~$u_2$ in~$H$. Finally, let
$v_3 = v_1 + k$, and let $u_3\in v_3 - \{b,b+k\}$ be the vertex
that \emph{precedes}~$v_3$ in~$H$. The choice of~$u_1$ implies
that $u_1$, $u_2$ and~$u_3$ do not all belong to the same
component of~$H$.

Let $w_1$ and~$w_2$ be the vertices that \emph{precede} $u_1$
and~$u_2$, respectively, on~$H$. (So $w_1 = w_2 + k$.)

Let $\sigma$ be the permutation of $\{1,2,3\}$ defined in the
proof of Lemma~\ref{parity-lemma}. If $\sigma$ is an even
permutation, let $H_1 = H$; if $\sigma$ is an odd permutation,
let $H_1$ be the element of~$\qc$ that has the same arcs as~$H$,
except:
 \begin{itemize}
 \item instead of the arcs from~$w_1$ to~$u_1$, and from~$w_2$
to~$u_2$, 
 \item there are arcs from~$w_1$ to~$u_2$, and from~$w_2$
to~$u_1$.
 \end{itemize}
 In either case, the permutation $\sigma_1$ for~$H_1$ is even.
Thus, $\sigma_1$ is either trivial or a $3$-cycle. If it is a
$3$-cycle, then $u_1$, $u_2$ and~$u_3$ are all contained in a
single component of~$H_1$, so $H_1$ has less components
than~$H$, which contradicts the minimality of~$H$. Thus,
$\sigma_1$ is trivial.

Let $H'$ be the element of~$\qc$ that has the same arcs
as~$H_1$, except:
 \begin{itemize}
 \item instead of the arcs from~$u_1$ to~$v_1$, from~$u_2$
to~$v_2$, and from~$u_3$ to~$v_3$, 
 \item there are arcs from~$u_1$ to~$v_2$, from~$u_2$ to~$v_3$,
and from~$u_3$ to~$v_1$.
 \end{itemize}
 Because $\sigma_1$ is trivial, we see that the permutation
$\sigma '$ for~$H'$ is the $3$-cycle $(1,2,3)$. Hence,  $u_1$,
$u_2$ and~$u_3$ are all contained in a single component of~$H'$,
so $H'$ has less components than~$H$, which contradicts the
minimality of~$H$.
\end{proof}

Thanks to \cref{IfabkNotG}, we may assume, henceforth, that $\langle a - b , k \rangle = G$. On the other hand, since $G$ is not cyclic, we have $\langle a - b \rangle \neq G$. Therefore, since $|k| = 2$, we conclude that 
	$$G = \langle a - b \rangle \oplus \langle k \rangle .$$
Since $G$ is not cyclic (and $|k| = 2$), this implies 
	$$ \text{$a - b$ has even order} .$$
It also implies that we may write 
	$$ \text{$a = a' + k'$ and $b = b' + k''$ for some (unique) $a',b' \in \langle a - b \rangle$ and $k', k'' \in \langle k \rangle$} . $$

\begin{lem} \label{NumCompsH0}
The number of connected components in~$H_0$ is
$$ \begin{cases}
 |G : \langle a, k \rangle| +  |G : \langle b, k \rangle|
& \text{if $k \notin \langle a \rangle$}  , \\
 |G : \langle b, k \rangle|
& \text{if $k \in \langle a \rangle$} .
\end{cases}$$
\end{lem}

\begin{proof} 
We consider two cases.

\setcounter{case}{0}

\begin{case}
Assume $k \notin \langle a \rangle$.
 \end{case}
For $i \in \{0,1\}$, let $G_i = \{\, v \in G \mid z_v
= i \,\}$, so each of $G_0$ and~$G_1$ has exactly half of the elements of~$G$.
From the definition of~$H_0$, we see that each component
of~$H_0$ is contained in either $G_0$ or~$G_1$. 
	\begin{itemize}
	\item Each component
in~$G_0$ is a cycle of length~$|a|$ (all $a$-arcs), so the
number of components in~$G_0$ is $|G_0|/|a| =(|G|/2)/|a| = |G : \langle a, k \rangle|$. 

	\item The number of components contained in~$G_1$ is equal to the
order of the quotient group $G/\langle b,k\rangle$. In other words, it is $|G : \langle b, k \rangle|$.
	\end{itemize}

\begin{case}
Assume $k \in \langle a \rangle$.
 \end{case}
 Let $xa + yb$ be a vertex that travels by~$a$ in~$H_0$. Then $v
= (|a|/2)a + yb$ is in the same component (by following a sequence
of $a$-arcs). Furthermore, if $y< |G:\langle a \rangle|-1$, then we see that
$x_{v+b} = |a|/2$, so $v$~travels by~$b+k$. Since $k = (|a|/2) a$, this means that
$(y+1)b = v + b + k$ is also in the same component. By induction
on~$y$, this implies that all the $a$-arcs of~$H_0$ are in the
same component, and this component contains some $(b+k)$-arcs.
Thus, the $a$-arcs are essentially irrelevant in counting
components of~$H_0$: there is a natural one-to-one
correspondence between the components of~$H_0$ and the
components of $\Cayd \bigl( G/ \langle k \rangle; b \bigr)$. Thus, the number of
components is $|G : \langle b, k \rangle|$.
 \end{proof}

\begin{lem} \label{odd-components}
$H_0$ has an odd number of connected components.
\end{lem}

\begin{proof}
Since $a' - b' \equiv a - b \pmod{\langle k \rangle}$, we have $\langle a' - b' \rangle = \langle a - b \rangle$, so one of~$a'$ and~$b'$ is an even multiple of $a - b$, and the other is an odd multiple. (Otherwise, the difference would be an even multiple of $a - b$, so it would not generate $\langle a - b \rangle$.)
 Thus, one of $|G : \langle a, k \rangle|$ and $|G : \langle b, k \rangle|$ is even, and the other is odd. So $|G : \langle a, k \rangle| + |G : \langle b, k \rangle|$ is odd. By \cref{NumCompsH0}, this establishes the desired conclusion if  $k \notin \langle a \rangle$.
 
We may now assume $k \in \langle a \rangle$.
This implies that the element $a'$ has odd order (and $k'$ must be nontrivial, but we do not need this fact). This means that $a'$ is an even multiple of $a-b$, so $b'$~must be an odd multiple of $a-b$ (since $\langle a' - b' \rangle = \langle a - b \rangle$). Therefore $|\langle a - b \rangle : \langle b' \rangle|$ is odd, which means $|G : \langle b,k \rangle|$ is odd. By \cref{NumCompsH0}, this establishes the desired conclusion.
\end{proof}

\begin{lem} \label{bk=G}
If $\langle b, k \rangle = G$, then $\Cayd(G; a, b, b+k)$ has a hamiltonian cycle.
\end{lem}

\begin{proof}
We may assume $\langle b \rangle \neq G$. (Otherwise, $(b)^{|G|}$ is a hamiltonian cycle.) Then, since $|k| = 2$ and $\langle b, k \rangle = G$, we must have $G = \langle b \rangle \oplus \langle k \rangle$. We may also assume $\langle b + k \rangle \neq G$. (Otherwise, $(b+k)^{|G|}$ is a hamiltonian cycle.) Since $G = \langle b \rangle \oplus \langle k \rangle$, this implies $|b|$~is even. So $(b,k)^{|b|}$ is a hamiltonian cycle.
\end{proof}

The following two lemmas complete the proof of \cref{3genAbelMustBeCyclic}.

\begin{lem} \label{boddindex}
If $|G : \langle b, k \rangle|$ is odd, then $\Cayd(G;a,b,b+k)$ has a hamiltonian cycle.
\end{lem}

\begin{proof}
From \cref{odd-components}, we know that $H_0$ has an odd number of
components. We construct a hamiltonian cycle by amalgamating
all of these components into one component. We start with the
component containing~$0$, and use \cref{amalgamate,amalgamate23} to add
the other components to it two at a time.

\setcounter{case}{0}

\begin{case} \label{boddindex-notin}
Assume $k \notin \langle a \rangle$.
 \end{case}
We may assume $\langle b, k \rangle \neq G$, for otherwise \cref{bk=G} applies. Since, by assumption, $|G : \langle b, k \rangle|$ is odd, this implies $|G : \langle b, k \rangle| \ge 3$. 

From the first paragraph of the proof of \cref{odd-components}, we know that $|G : \langle a, k \rangle| + |G : \langle b, k \rangle|$ is odd. Since $|G : \langle b, k \rangle|$ is also odd, this implies that $|G : \langle a, k \rangle|$ is even. For convenience, let
	$ m = |G : \langle a, k \rangle|$.

 Note that two vertices $u$~and~$v$ are in the same component
of~$H_0$ if and only if either
 \begin{itemize}
 \item $z_u = z_v = 0$ and $y_u = y_v$; or
 \item $z_u = z_v = 1$ and $x_u \equiv x_v \pmod{|G : \langle b , k \rangle|}$. 
 \end{itemize}

Lemma~\ref{amalgamate} implies there is an element~$H_0'$
of~$\qf$, such that $0$, $k$, and~$a+k$ are all in the same
component of~$H_0'$. (The other components of~$H_0'$ are
components of~$H_0$.) 

Then Lemma~\ref{amalgamate} implies there is an element~$H_1 =
(H_0')'$ of~$\qf$, such that $a+b$, $a+b+k$, and~$2a+b+k$ are
all in the same component of~$H_1$. (The other components
of~$H_1$ are components of~$H_0$.) 

With this as the base case of an inductive construction, we
construct, for $1 \le i \le m/2$, an element~$H_i$ of~$\qf$,
such that 
 $$ \{\, v \mid
 \text{$z_v = 0$ and $0 \le y_v \le 2i-1$}\,\}
 \cup 
 \{\, v \mid
 \text{$z_v = 1$ and $x_v \equiv
 0,1,\text{ or } 2 \pmod {|G : \langle b, k \rangle|}$} \,\}
 $$
 is a component of~$H_i$, and all other components of~$H_i$ are
components of~$H_0$. Namely:
	\begin{itemize} \itemsep=\smallskipamount
	\item We apply \cref{amalgamate23} to $H_{i-1}$ with $u = (2i-2)b$, to obtain $H_{i-1}' \in \qf$, such that $(2i-2)b$ and $(2i-1)b + k$ are in the same component, but $a + (2i-2)b + k$ is in a different component. Specifically, the following arcs are removed from~$H_{i-1}$:
	$$ \begin{array}{rrrrrrrrrrr}
  & &(2i-2)b & &   &\to&   a&+& (2i-2)b \\
  & &(2i-2)b &+& k &\to&    & & (2i-1)b &+& k \\
 a&+&(2i-3)b &+& k &\to&   a&+& (2i-2)b &+& k 
 	\hbox to 0pt{.\hss} \end{array}
	$$
In their place, the following arcs are inserted into~$H_{i-1}'$:
 $$ \begin{array}{rrrrrrrrrrr}
  & &(2i-2)b & &   &\to&    & & (2i-1)b &+& k \\
  & &(2i-2)b &+& k &\to&   a&+& (2i-2)b &+& k \\
 a&+&(2i-3)b &+& k &\to&   a&+& (2i-2)b 
\hbox to 0pt{.\hss} \end{array} $$
	Note that the two vertices $a + (2i-2)b + k$ and $a + (2i-1)b + k$ are in the same component of~$H_{i-1}'$. (Indeed, they are adjacent, since $a + (2i-2)b + k$ travels by~$b$.)

	\item Then $(2i-1)b$, $(2i-1)b + k$, and $a + (2i-1)b + k$ are all in different components of $H_{i-1}'$, so \cref{amalgamate} yields $H_i = (H_{i-1}')'$, such that all three of these vertices are in the same component. (And all other components are components of~$H_0$.)
	Specifically, the following arcs are removed from~$H_{i-1}'$:
 $$ \begin{array}{rrrrrrrrrrr}
  & &(2i-1)b & &   &\to&   a&+& (2i-1)b \\
  & &(2i-1)b &+& k &\to&    &  v & = \hfill  (2i)b & \text{or} & (2i)b + k \\
 a&+&(2i-2)b &+& k &\to&   a&+& (2i-1)b &+& k \hbox to 0pt{.\hss}\hfil\hfil  \\
\hbox to 0pt{.\hss} \end{array} $$
In their place, the following arcs are inserted into~$H_i$:
 $$ \begin{array}{rrrrrrrrrrr}
  & &(2i-1)b & &   &\to&    & & v\hfil \\
  & &(2i-1)b &+& k &\to&   a&+& (2i-1)b &+& k \\
 a&+&(2i-2)b &+& k &\to&   a&+& (2i-1)b 
 \hbox to 0pt{.\hss}\end{array} $$
 	\end{itemize}

Let $K_1 = H_{m/2}$. With this as the base case of an
inductive construction, we construct, for $1 \le i \le
\bigl(|G : \langle b, k \rangle|-1 \bigr)/2$, an element~$K_i$ of~$\qf$, such that 
 $$ \{\, v \mid z_v = 0 \,\}
 \cup 
 \{\, v \mid
 \text{$z_v = 1$ and $x_v \equiv
 0,1, \ldots, \text{ or~$2i$} \pmod {|G : \langle b, k \rangle|}$} \,\}
 $$
 is a component of~$K_i$, and all other components of~$K_i$ are
components of~$H_0$. Namely, Lemma~\ref{amalgamate} implies
there is an element~$K_i = K_{i-1}'$ of~$\qf$, such that
$(2i-1)a$, $(2i-1)a+k$, and $(2i)a + k$ are all in the same
component of~$K_i$. 

Then, for $i = \bigl(|G : \langle b, k \rangle|-1 \bigr)/2$, we see that a single
component of~$K_i$ contains every vertex, so $K_i$ is a
hamiltonian cycle.

\begin{case}
Assume $k \in \langle a \rangle$.
 \end{case}
 Note that one component of~$H_0$ is
 $$ \{\, v \mid x_v < |a|/2 \,\}
 \cup
 \{ \, v \mid x_v \equiv 0 \pmod{|G : \langle b, k \rangle|} \, \} .$$
 Two vertices $u$~and~$v$ that are not in this component are in
the same component of~$H_0$ if and only if $x_u \equiv x_v
\pmod{|G : \langle b, k \rangle|}$.

We may assume $\langle a \rangle \neq G$, for otherwise $(a)^{|a|}$ is a hamiltonian cycle. With $H_0$ as the base case of an inductive
construction, we construct, for $0 \le i \le \bigl(|G : \langle b, k \rangle|-1 \bigr)/2$, an element~$H_i$ of~$\qf$, such that 
 $$ \{\, v \mid x_v < |a|/2 \,\}
 \cup
 \{ \, v \mid x_v \equiv 
 0,1, \ldots, \text{ or~$2i$} \pmod{|G : \langle b, k \rangle|} \,\}
 $$
 is a component of~$H_i$, and all other components of~$H_i$ are
components of~$H_0$. Namely, Lemma~\ref{amalgamate} implies
there is an element~$H_i = H_{i-1}'$ of~$\qf$, such that
$(2i-1)a$, $(2i-1)a+k$, and $(2i)a + k$ are all in the same
component of~$H_i$. 

Then, for $i = \bigl(|G : \langle b, k \rangle|-1\bigr)/2$, we see that a single
component of~$H_i$ contains every vertex, so $H_i$ is a
hamiltonian cycle.
 \end{proof}

\begin{lem}
If $|G : \langle b, k \rangle|$ is even, then $\Cayd(G;a,b,b+k)$ has a hamiltonian cycle.
\end{lem}

\begin{proof}
For convenience, let $m = |G : \langle a, k \rangle|$. Since $|G : \langle b, k \rangle|$ is even, and \cref{odd-components} tells us that $H_0$ has an odd number of
components, we see from \cref{NumCompsH0} that 
	$$ \text{$k \notin \langle a \rangle$ and $m$ is odd} .$$ 

Define $H_0'$ as in \cref{boddindex-notin} of the proof of \cref{boddindex}, and let $H_1 = H_0'$.
With this as the base case of an inductive construction, we
inductively construct, for $1 \le i \le (m+1)/2$, an element~$H_i$ of~$\qf$,
such that 
$$ \{\, v \mid \text{$z_v = 0$ and $0 \le y_v \le 2i-2$} \,\}
	\cup
	\{\, v \mid \text{$z_v = 1$ and $x_v = 0$ or~$1$ (mod $|G : \langle b, k \rangle|$)} \,\} $$
is a component of~$H_i$, and all other components are components of~$H_0$. 
Namely:
	\begin{itemize} \itemsep=\smallskipamount
	\item We apply \cref{amalgamate23} to $H_{i-1}$ with $u = (2i-3)b$, to obtain $H_{i-1}' \in \qf$, such that $(2i-3)b$ and $(2i-2)b + k$ are in the same component, but $a + (2i-3)b + k$ is in a different component. Specifically, the following arcs are removed from~$H_{i-1}$:
	$$ \begin{array}{rrrrrrrrrrr}
  & &(2i-3)b & &   &\to&   a&+& (2i-3)b \\
  & &(2i-3)b &+& k &\to&    & & (2i-2)b &+& k \\
 a&+&(2i-4)b &+& k &\to&   a&+& (2i-3)b &+& k 
 	\hbox to 0pt{.\hss} \end{array}
	$$
In their place, the following arcs are inserted into~$H_{i-1}'$:
 $$ \begin{array}{rrrrrrrrrrr}
  & &(2i-3)b & &   &\to&    & & (2i-2)b &+& k \\
  & &(2i-3)b &+& k &\to&   a&+& (2i-3)b &+& k \\
 a&+&(2i-4)b &+& k &\to&   a&+& (2i-3)b 
\hbox to 0pt{.\hss} \end{array} $$
Note that the two vertices $a + (2i-3)b + k$ and $a + (2i-2)b + k$ are in the same component of~$H_{i-1}'$ (indeed, they are adjacent).

	\item Then $(2i-2)b$, $(2i-2)b + k$, and $a + (2i-2)b + k$ are all in different components of $H_{i-1}'$, so \cref{amalgamate} yields $H_i = (H_{i-1}')'$, such that all three of these vertices are in the same component. (And all other components are components of~$H_0$.)
	Specifically, the following arcs are removed from~$H_{i-1}'$:
 $$ \begin{array}{rrrrrrrrrrr}
  & &(2i-2)b & &   &\to&   a&+& (2i-2)b \\
  & &(2i-2)b &+& k &\to&    &  v & = \   (2i-1)b & \text{or} & (2i-1)b + k  \\
 a&+&(2i-3)b &+& k &\to&   a&+& (2i-2)b &+& k \hbox to 0pt{.\hss}\hfil\hfil 
 \end{array} $$
In their place, the following arcs are inserted into~$H_i$:
 $$ \begin{array}{rrrrrrrrrrr}
  & &(2i-2)b & &   &\to&    & & v\hfil \\
  & &(2i-2)b &+& k &\to&   a&+& (2i-2)b &+& k \\
 a&+&(2i-3)b &+& k &\to&   a&+& (2i-2)b 
 \hbox to 0pt{.\hss}\end{array} $$

	\end{itemize}

Let $K_1 = H_{(m+1)/2}$. With this as the base case of an
inductive construction, we construct, for $1 \le i \le
|G : \langle b, k \rangle|/2$, an element~$K_i$ of~$\qf$, such that 
 $$ \{\, v \mid z_v = 0 \,\}
 \cup 
 \{\, v \mid
 \text{$z_v = 1$ and $x_v \equiv
 0,1, \ldots, \text{ or~$2i - 1$} \pmod {|G : \langle b, k \rangle|}$} \,\}
 $$
 is a component of~$K_i$, and all other components of~$K_i$ are
components of~$H_0$. Namely, Lemma~\ref{amalgamate} implies
there is an element $K_i = K_{i-1}'$ of~$\qf$, such that
$(2i-1)a$, $(2i-1)a+k$, and $(2i)a + k$ are all in the same
component of~$K_i$. 

Then, for $i = |G : \langle b, k \rangle|/2$, we see that a single
component of~$K_i$ contains every vertex, so $K_i$ is a
hamiltonian cycle.
\end{proof}

\end{document}